\documentclass[12pt]{amsart}
\usepackage{latexsym}
\usepackage[utf8]{inputenc}
\usepackage{amsmath}
\usepackage{amsthm}
\usepackage{amsfonts}
\usepackage{amssymb}
\usepackage{epsfig}
\usepackage{nicefrac}
\usepackage[all]{xy}
\usepackage{graphicx}
\usepackage{enumerate}
\usepackage[T1]{fontenc}
\usepackage[font=small,skip=0pt]{caption}
\usepackage{tikz}
\usepackage{subfig}
 
\newtheorem{theorem}{Theorem}[section]
\newtheorem{corollary}[theorem]{Corollary}
\newtheorem{proposition}[theorem]{Proposition}
\newtheorem{lemma}[theorem]{Lemma}

\newtheorem{question}[theorem]{Question}

\theoremstyle{definition}

\newtheorem{remark}[theorem]{Remark}
\newtheorem{example}[theorem]{Example}

\newtheorem{definition}[theorem]{Definition}

\newcommand{\CC}{{\mathbb C}}

\newcommand{\RR}{{\mathbb R}}
\newcommand{\ZZ}{{\mathbb Z}}
\newcommand{\NN}{{\mathbb N}}

\begin{document}

\title{Packing Lagrangian tori}
\date{\today}
\author{Richard K. Hind}
\address{RH: Department of Mathematics\\
University of Notre Dame\\
255 Hurley\\
Notre Dame, IN 46556, USA.}
\author{Ely Kerman}
\address{EK: Department of Mathematics\\
University of Illinois at Urbana-Champaign\\
1409 West Green Street\\
Urbana, IL 61801, USA.}
\thanks{Both authors are supported by grants from the Simons Foundation}

\maketitle

\begin{section}{Introduction.}
In this paper we consider packings of symplectic manifolds by Lagrangian tori. Since every symplectic manifold contains infinitely many disjoint Lagrangian tori, we must set a scale in order to pose meaningful questions. We therefore restrict our attention to Lagrangian tori whose area homomorphism takes only integer values. These will be referred to as {\em integral Lagrangian tori}.\footnote{These are also sometimes called Bohr-Sommerfeld Lagrangians.} The fundamental  packing question, in this setting, is the following.

\bigskip

\noindent \textit{What is the maximum number of disjoint integral Lagrangian tori contained in a given (pre)compact symplectic manifold?}

\bigskip

A more approachable version of this question is to consider a specific collection of disjoint integral Lagrangian tori in a symplectic manifold $(M, \omega)$, and to ask if it is a {\em maximal integral packing} in the sense that any other integral Lagrangian torus in $M$ must intersect at least one torus in the collection. In this paper, we study  this question in the simplest nontrivial setting.

\subsection{Results}

Equip the sphere $S^2$ with its standard symplectic form $\omega$ scaled so that $\int_{S^2} \omega =2$. Let $L_{1,1}$ be the monotone Clifford torus
(product of equators) in $(S^2 \times S^2, \pi_1^*\omega + \pi_2^*\omega).$ Our first result is the following.

\begin{theorem}\label{one} The Clifford torus $L_{1,1}$ is a maximal integral packing of $(S^2 \times S^2, \pi_1^*\omega + \pi_2^*\omega).$
\end{theorem}

For real numbers $a,\,b>0$, consider the symplectic polydisk 
$$P(a,b) =\left\{(z_1, z_2) \in \CC^2 \, \mid \, \pi|z_1|^2<a,\,\pi|z_2|^2<b\right\} \subset \RR^4 .$$
Identifying $L_{1,1}$ with the standard Clifford torus in $\RR^4$, Theorem \ref{one} implies that $L_{1,1}$ is a maximal integral packing of each $P(a,b)$
with $1<a,\,b < 2$.

If $a$ and $b$ are both greater then $2$, then a natural candidate for a maximal integral  packing of $P(a,b)$ is the collection of integral Lagrangian tori 
$$\{L_{k,l} \mid k,l \in \NN, k \leq \lfloor a \rfloor,\, l \leq  \lfloor b \rfloor\},$$ where $L_{k,l}$ is the product of the circle about the origin bounding area $k$ in the $z_1$-plane with the circle about the origin bounding area $l$ in the $z_2$-plane. The analogous packing in dimension two is always maximal. Our second result shows that, in dimension four, this candidate always fails.

\begin{theorem}\label{two} If $\min(a,b) >2$, then  $\{L_{k,l} \mid k,l \in \NN, k \leq \lfloor a \rfloor,\, l \leq  \lfloor b \rfloor\}$ is not a maximal integral packing of $P(a,b)$. For every $\epsilon>0$ there is an integral Lagrangian torus $L^+$ in $$P(2+\epsilon, 2+\epsilon) \smallsetminus \{L_{k,l} \mid k,l \in \{1,2\}\}.$$
\end{theorem}

\subsection{Overview}

The first step in our proof of Theorem \ref{one} is to show that any integral Lagrangian torus in $(S^2 \times S^2, \pi_1^*\omega + \pi_2^*\omega)$ is actually monotone. This follows from the work of Hind and Opshtein in \cite{ho}, and is proved in  Lemma \ref{monotone} below. Arguing by contradiction, we then assume there is a monotone Lagrangian torus $\mathbb{L}$ in $(S^2 \times S^2, \pi_1^*\omega + \pi_2^*\omega)$ that is disjoint from the Clifford torus $L_{1,1}$.  
%
%
The work of Ivrii in \cite{ivriit}, and Dimitroglou-Rizell, Goodman, and Ivrii in \cite{rgi}, implies that there is a finite energy holomorphic foliation $\mathcal{F}$ of $S^2 \times S^2 \smallsetminus (\mathbb{L} \cup L_{1,1})$ which has a normal form near $\mathbb{L}$ and $L_{1,1}$ (see Section \ref{double}). 
We use $\mathcal{F}$ to establish  the existence of two symplectic spheres, $F$ and $G$,  in $(S^2 \times S^2, \pi_1^*\omega + \pi_2^*\omega)$ (see Proposition \ref{existence2} and Proposition \ref{existence3}).  Both $F$ and $G$ represent a homology class of the form $(1,d) \in H_2(S^2 \times S^2; \ZZ)=\ZZ^2$ for some large $d$. They also have special intersection properties with the leaves of $\mathcal{F}$ and with each other (see Section \ref{FG} and Proposition \ref{intcount}). Using the spheres $F$ and $G$, together with the operations of blow-up, inflation, and blow-down, we then alter the ambient symplectic manifold away from $\mathbb{L} \cup L_{1,1}$ to obtain a new monotone symplectic manifold, $(X, \Omega)$.  This new manifold is symplectomorphic to $(S^2 \times S^2, (d+1)(\pi_1^* \omega + \pi_2^*\omega))$ and $\mathbb{L}$ and $L_{1,1}$ remain disjoint and monotone therein. However the images (transforms) of the spheres $F$ and $G$ in  $(X, \Omega)$ are now in the class $(1,0)$ and their existence implies, by the work of Cielieback and Schwingenheur in \cite{cisc}, that $\mathbb{L}$ and $L_{1,1}$ must both be Hamiltonian isotopic to the Clifford torus in $(X, \Omega)$. It then follows from  standard monotone Lagrangian Floer thoery, \cite{oh}, that it is not possible for  $\mathbb{L}$ and $L_{1,1}$ to be disjoint. This  contradiction completes the proof of Theorem \ref{one}.

To prove Theorem \ref{two} we construct, for every $\epsilon>0$, an explicit embedding of the closure of $P(1,1)$ into $P(2+\epsilon, 2+\epsilon) \smallsetminus \{L_{k,l} \mid k,l \in \{1,2\}\}$, using a time-dependent Hamiltonian flow. The desired Lagrangian, $L^+$, is the one on the boundary of the image. 

\subsection{Commentary and further questions}

Given that  Theorem \ref{one} is reduced to the problem of detecting  intersection points of two monotone Lagrangian tori, using \cite{ho}, it is natural to ask whether Lagrangian Floer theory (rigid holomorphic curves) can also be used to prove Theorem \ref{one} directly. To the knowledge of the authors this is not yet possible. The following result seems to be as close to a proof of Theorem \ref{one} as one can currently get using Lagrangian Floer theory.

\begin{theorem}
\label{lag}
Suppose that $L$ is a monotone Lagrangian torus in $(S^2 \times S^2, \pi_1^*\omega + \pi_2^*\omega)$. If the Lagrangian Floer homology of $L$, with respect to some $\CC^*$-local system, is nontrivial, then $L$ must intersect $L_{1,1}$.
\end{theorem}
This follows from work of Evans and Lekili in \cite{evle} which implies that the Clifford torus split-generates the monotone Fukaya category of $(S^2 \times S^2, \pi_1^*\omega + \pi_2^*\omega)$. It is not known whether there exist monotone Lagrangian tori in 
$(S^2 \times S^2, \pi_1^*\omega + \pi_2^*\omega)$ whose Lagrangian Floer homology is trivial for every choice of $\CC^*$-local system.
In \cite{vianna}, Vianna constructs a countably infinite collection of monotone Lagrangian tori in $(S^2 \times S^2, \pi_1^*\omega + \pi_2^*\omega)$, no two of which are Hamiltonian isotopic. Each of the tori in Vianna's collection satisfies the hypothesis of Theorem \ref{lag}.

The following question, in the sprit of Theorem \ref{one}, remains unresolved.
\begin{question}\label{always}
Does every pair of monotone Lagrangian tori in $(S^2 \times S^2, \pi_1^*\omega + \pi_2^*\omega)$ intersect?
\end{question}

Progress on other aspects of the study of disjoint Lagrangian tori has also recently been made in two related works by Mak and Smith, in \cite{maksmi}, and by Polterovich and Shelukhin, in \cite{polshe}. Let $\{\gamma_i\}$ be a collection of disjoint circles bounding disks of the same area, and let $E$ be the equator in the sphere $S^2$. In \cite{maksmi} and \cite{polshe}, it is shown that, with respect to certain nonmonotone symplectic forms on $S^2 \times S^2$, packings of the form ${\mathcal L} = \sqcup \gamma_i \times E$ are maximal in the sense that any Lagrangian torus Hamiltonian isotopic to $\gamma_1 \times E$ must intersect ${\mathcal L}$. In comparison, the maximal packing given by Theorem \ref{one} only includes a single torus, $L_{1,1}$, but we do not assume any other tori are Hamiltonian isotopic to it. Theorem \ref{two} shows that analogous packings of the form $\sqcup \gamma_i \times \gamma_j$ are no longer maximal.

Below are a few of the questions suggested by Theorem \ref{two} which also remain unresolved. 

\begin{question}
Is every integral Lagrangian torus in $P(2+\epsilon, 2+\epsilon) \smallsetminus \{L_{k,l} \mid k,l \in \{1,2\}\}$, Hamiltonian isotopic to $L_{1,1}$?
\end{question}

\begin{question}
Suppose $2<a,\,b <3$.  Are there six disjoint integral Lagrangian tori in $P(a,b)$?
\end{question}

\begin{question}
Suppose $2<b <3$.  Are there three disjoint integral Lagrangian tori in $P(2,b)$?
\end{question}

\end{section}

\section{Conventions, labels and notation}

Every copy of the two dimensional sphere $S^2$ will implicitly be identified with the unit sphere in $\RR^3$ and we will label the north and south poles  by $\infty$ and $0$, respectively. In 
$(S^2 \times S^2, \pi_1^*\omega + \pi_2^*\omega)$, we use these points to define the four symplectic spheres
$S_0=S^2 \times \{0\}$, $S_{\infty}=S^2 \times \{\infty\}$, $T_0=\{0\} \times S^2$ and $T_{\infty}=\{\infty\} \times S^2$. The ordered basis $\{ [S_0], [T_0]\}$ of $H_2(S^2 \times S^2;\ZZ)$ is used to identify it with $\ZZ^2$.

Let $L \subset (M, \Omega)$ be a Lagrangian torus in a four dimensional symplectic manifold. A diffeomorphism $\psi$ from $\mathbb{T}^2= S^1 \times S^1$ to $L$ will be referred to as a parameterization of $L$.  It specifies a basis of  $H_1(L; \ZZ)$ and thus an isomorphism from $H_1(L; \ZZ)$ to $\ZZ^2$. We will denote this copy of $\ZZ^2$ by $H_1^{\psi}(L; \ZZ)$. The parameterization $\psi$ can also be extended to a  symplectomorphism $\Psi$ from a neighborhood of the zero section in $T^*\mathbb{T}^2$ to a Weinstein neighborhood $\mathcal{U}(L)$ of $L$ in $M$.  We will denote the corresponding coordinates in the neighborhood  $\mathcal{U}(L)$ of  $L$ by $(p_1, p_2, q_1, q_2)$ and, for simplicity, we will assume that $$\mathcal{U}(L) =  \{|p_1| < \epsilon,   |p_2 |<  \epsilon\},$$ for some $\epsilon >0$. 

\section{Proof of Theorem \ref{one}.}\label{curves}

Arguing by contradiction, we begin with the following.

\medskip

\noindent{\bf Assumption 1.}
There is an integral Lagrangian torus $\mathbb{L}$ in $(S^2 \times S^2, \pi_1^*\omega + \pi_2^*\omega)$ which is disjoint from the Clifford torus $L_{1,1}$.

\medskip

Below we show that  Assumption 1 can be refined in three ways. 

\subsection{Refinement 1. We may assume that $\mathbb{L}$ is monotone.} 

A symplectic manifold $(M, \Omega)$ is {\em monotone} if the Chern and area homomorphisms,
$$
c_1 \colon \pi_2(M) \subset H_2(M, \ZZ) \to \ZZ  \text{ and } \Omega  \colon \pi_2(M) \to \RR,
$$
are positively proportional. Recall that a Lagrangian submanifold $L \subset (M, \Omega)$ is {\em monotone} if its Maslov and area homomorphisms,
$$
\mu \colon \pi_2(M,L) \to \ZZ  \text{ and } \Omega  \colon \pi_2(M,L) \to \RR,
$$
are positively proportional. We will denote the constant of proportionality of $L$ by $\lambda$.

If $L$ is a Lagrangian torus, then  one can verify monotonicity by checking it for a collection of disks whose boundaries generate $H_1(L; \ZZ)$.

\begin{lemma}
\label{pair}
Suppose that $(M, \Omega)$ is a symplectic four manifold which is monotone with constant $\frac{\lambda}{2}$. A Lagrangian torus $L$ in  $(M, \Omega)$ is monotone with constant $\lambda$ if there are two smooth maps $v_1,\,v_2 \colon (D^2, S^1) \to (M,L)$ such that the boundary maps $v_1|_{S^1}$ and $v_2|_{S^1}$ determine an integral basis of $H_1(L;\ZZ) =\ZZ^2$ and $\mu([v_i]) =\lambda \Omega([v_i])$ for $i=1,2.$
\end{lemma}

Refinement 1 is validated by the following result.

\begin{proposition}\label{monotone} 
Every integral Lagrangian torus $L$ in $(S^2 \times S^2, \pi_1^*\omega + \pi_2^*\omega)$ is monotone.
\end{proposition}

\begin{proof} 
By Theorem C of \cite{rgi} there is a Hamiltonian diffeomorphism which displaces $L$ from  the pair of spheres $S_{\infty} \cup T_{\infty}.$ Hence, $L$ can be identified with an integral Lagrangian torus $\mathbf{ L}$ inside the polydisk $P(2-\epsilon,2-\epsilon) \subset (\RR^4, \omega_4)$ for some sufficiently small $\epsilon>0.$ By Lemma \ref{pair}, it suffices to find two smooth maps $v_1,\,v_2 \colon (D^2, S^1) \to (\RR^4, \mathbf{ L})$ such that the boundary maps $v_1|_{S^1}$ and $v_2|_{S^1}$ determine an integral basis of $H_1(\mathbf{L};\ZZ)$ and $\mu([v_i]) =2 \omega_4([v_i])$ for $i=1,2.$ Simplifying further, we note that, for $\RR^4$, the maps $\mu$ and $\omega_4$ can be recast as homomorphisms $$\mu \colon H_1(\mathbf{ L}; \ZZ) \to \ZZ \text{ and }\omega_4 \colon H_1(\mathbf{ L}; \ZZ) \to \RR$$ and it suffices to find an integral basis $\{e_1,e_2\}$ of $H_1(\mathbf{ L}; \ZZ)$  such that $\mu(e_i) =2 \omega_4(e_i)$ for $i=1,2$.

Since $\mathbf{ L}$ is contained in $P(2-\epsilon,2-\epsilon)$, it follows from \cite{cm} that there is a smooth  map  $f \colon (D, S^1) \to (\RR^4, \mathbf{ L})$ of  Maslov index $2$ whose symplectic area is positive and less than two. Since $\mathbf{ L}$ is integral, the area of $f$ must be equal to one. If  $e_1$ is the element of $H_1(\mathbf{ L}; \ZZ)$ represented by $f|_{S^1}$, then we have $\mu(e_1)=2$ and $\omega_4(e_1)=1$. 

Let $c$ be a class in $H_1(\mathbf{ L}; \ZZ)$ such that $\{e_1,c\}$ is an integral basis. Since $\mu(c)$ is even, by adding integer multiples of $e_1$ to $c$, if necessary, we may assume that $\mu(c)=2$. It remains to show that $\omega_4(c) =1.$

Arguing by contradiction, assume that $\omega_4(c) \neq 1$. Set 
$$
\hat{c}=
\begin{cases}
   c   &  \text{if}\,\, \omega_4(c) > 1,\\
   c + 2(e_1-c) & \text{if}\,\, \omega_4(c) < 1.
\end{cases}
$$
Then $\{e_1, \hat{c}\}$ is an integer basis of $H_1(\mathbf{ L}; \ZZ)$ 
that satisfies
$$\omega_4(e_1)=1,\,\omega_4(\hat{c}) \geq 2$$
and
 $$\mu(e_1)=\mu(\hat{c}) =2.$$
In \cite{ho}, Hind and Opshtein prove that if a Lagrangian torus in $P(a,b)$ admits such a basis, then either $a>2$ or $b>2$. 
This contradicts the assumption that $\mathbf{ L}$ lies in $P(2-\epsilon,2-\epsilon)$ and we are done.
\end{proof}

\subsection{Refinement 2: We may assume that $\mathbb{L}$ lies in the complement of $S_0 \cup S_{\infty} \cup T_0 \cup T_{\infty}$} To verify this, we utilize the relative finite energy foliations from \cite{rgi} which we now recall.  

\subsubsection{Foliations of $(S^2 \times S^2) \smallsetminus L$ following \cite{rgi} }
In \cite{gr}, Gromov proves that for any smooth almost-complex structure $J$ on $S^2 \times S^2$ that is tamed by the symplectic form $\pi_1^*\omega + \pi_2^*\omega$, there is a foliation of $S^2 \times S^2$ by $J$-holomorphic spheres in the class $(0,1)$ (and another with fibres in the class $(1,0)$).  For any monotone Lagrangian torus $L \subset (S^2 \times S^2, \pi_1^*\omega + \pi_2^*\omega)$, there is an analogous relative theory, developed first by Irvrii in \cite{ivriit}, and completed in Dimitroglou-Rizell, Goodman and  Ivrii in \cite{rgi}, with input from \cite{we} and \cite{hl}. This  yields symplectic $S^2$-foliations of  $S^2 \times S^2$ that are {\em compatible with $L$}. These are obtained by stretching certain Gromov foliations along $L$ and smoothing the compactifications of the limiting buildings with more than one level. We now describe a version of this theory that has been adapted for the purposes of this paper. As in \cite{hl},  we only consider the curves which, after stretching, map to $S^2 \times S^2 \smallsetminus L$.  
\medskip

\noindent{\bf Input.} Let $L$ be a monotone Lagrangian torus in $(S^2 \times S^2, \pi_1^*\omega + \pi_2^*\omega)$. 
Fix a parameterization $\psi$ of $L$ and the corresponding Weinstein neighborhood $$\mathcal{U}(L)=\{|p_1|<\epsilon, |p_2|<\epsilon\}.$$  Fix a tame almost complex structure $J$ on ($S^2 \times S^2 \smallsetminus L, \pi_1^*\omega + \pi_2^*\omega)$ such that in $\mathcal{U}(L)$ we have
$$J \frac{\partial}{\partial q_i} =  - \sqrt{p_1^2+p_2^2} \, \frac{\partial}{\partial p_i}.$$

Recall that each negative end of a finite energy $J$-holomorphic curve $u$ mapping to $S^2 \times S^2 \smallsetminus L$  is asymptotic to a closed Reeb orbit on a copy of the flat unit cotangent bundle, $S_L^*\mathbb{T}^2$, of $\mathbb{T}^2$, corresponding to $L$.   This Reeb orbit  covers a closed geodesic, $\gamma$, of the flat metric on $\mathbb{T}^2$. In this case,  we simply say that the end of $u$ is asymptotic to $L$ along $\gamma$.

\medskip

\noindent{\bf Output.} From this input, one can construct, as in $\S 2.5$ of \cite{rgi}, a family of almost complex structures $J_{\tau}$ on $S^2 \times S^2$ for $\tau \geq 0$. Taking the limit of the Gromov foliations for the $J_{\tau}$ as $\tau \to \infty$, it follows from Theorem D  and Proposition 5.16 of \cite{rgi} that one obtains a foliation $\mathcal{F}=\mathcal{F}(L, \psi, J)$ of $S^2 \times S^2\smallsetminus L$ with the following properties.

\begin{itemize}
\item The foliation $\mathcal{F}$ has two kind of leaves: unbroken ones consisting of a single closed $J$-holomorphic sphere in $S^2 \times S^2\smallsetminus L$ of class $(0,1)$, and  broken leaves consisting of a pair of  finite energy $J$-holomorphic planes in $S^2 \times S^2\smallsetminus L$. 
  
\item Each leaf of $\mathcal{F}$ intersects $S_{\infty}$ in exactly one point. For a broken leaf this means that exactly one of its planes intersects $S_{\infty}$. 
  
\item The ends of two planes of a broken leaf are asymptotic to the same  geodesic, but with opposite orientations. This geodesic is embedded. We denote its homology class, equipped with the orientation determined by the plane which intersects $S_{\infty}$, by $\beta \in H_1(L; \ZZ)$. This class is the same for all broken leaves of $\mathcal{F}$ and is referred to as the  foliation class of $\mathcal{F}$.  
  
  
\item Each point $z \in S^2 \times S^2 \smallsetminus L$ lies in a unique leaf of $\mathcal{F}$, and each point of $L$ lies on a unique geodesic in the foliation class $\beta$ that corresponds to a unique plane of a broken leaf of $\mathcal{F}$  that intersects $S_{\infty}$.

\item Let $p \colon S^2 \times S^2 \to S_{\infty}$ be the map which takes $z \in S^2 \times S^2 \smallsetminus L$ to the intersection of its leaf with $S_\infty$, and takes $z \in L$ to the intersection with $S_{\infty}$ of the broken leaf asymptotic to the unique geodesic through $z$ representing the foliation class. The map is well defined by positivity of intersection since $S_{\infty}$ is complex. Then $p(L)$ is an embedded closed curve in $S_{\infty}$. Moreover, if $L$ is homotopic to $L_{1,1}$ in the complement of $T_0 \cup T_{\infty}$, then $p(T_0)$ and $p(T_{\infty})$ (which are points since $T_0$ and $T_{\infty}$ are complex) lie on opposite sides of the closed curve $p(L)$.

\item If $L$ is disjoint from $S_0 \cup S_{\infty} \cup T_0 \cup T_{\infty}$, then we may assume this configuration of symplectic spheres is $J$-complex

\end{itemize}

\begin{lemma}(Straightening).\label{straight} For all sufficiently small $\epsilon>0$ we may assume, by perturbing $J$ outside of 
$\mathcal{U}(L)$, that the unbroken leaves of $\mathcal{F}$ that intersect $\mathcal{U}(\mathbb{L})$ do so along the annuli $\{p_1=\delta, q_1 = \theta, -\epsilon< p_2<\epsilon\}$ for some $\theta \in S^1$ and nonzero $\delta \in (-\epsilon, \epsilon)$. As well, the planes of broken leaves of $\mathcal{F}$  through $S_{\infty}$ intersect $\mathcal{U}(\mathbb{L})$ along the annuli 
$\{p_1=0, q_1 = \theta, 0<p_2<\epsilon\}$, for some $\theta \in S^1$, and the planes of $\mathcal{F}$  through  $S_0$ intersect $\mathcal{U}(\mathbb{L})$ along the annuli 
$\{p_1=0, q_1 = \theta, -\epsilon<p_2<0\}$.
\end{lemma}

\begin{proof} The statement for broken leaves was established in Proposition 5.16 of \cite{rgi} (see the first bullet point of the proof). Given this, our unbroken leaves intersect $\partial \mathcal{U}(L)$ close to $p_1=0$ in circles smoothly approximating circles $\{p_1 = \delta, q_1 = \theta, p_2 = \pm \epsilon\}$. We look first at the parts of these leaves mapping to the complement of $\mathcal{U}(L)$, which are families of holomorphic disks. Perturbing the disks  we may assume they intersect $\partial \mathcal{U}(L)$ close to $p_1=0$ precisely in the circles $\{p_1 = \delta, q_1 = \theta, p_2 = \pm \epsilon\}$ while remaining symplectic and smoothly converging to the broken leaves. Hence changing $J$ outside of $\mathcal{U}(L)$ we may assume the perturbed disks remain holomorphic. These new disks match with the annuli $\{p_1=\delta, q_1 = \theta, -\epsilon< p_2<\epsilon\}$ to give holomorphic spheres in the class $(0,1)$, and in fact by positivity of intersection these are the only spheres in the class intersecting the region $\{|p_1| < \epsilon\}$, at least if we shrink $\epsilon$ to include only the region where the perturbations apply.
\end{proof}

\medskip

\noindent{\bf Solid tori.}\label{solid} In the case when $L$ is disjoint from $S_0 \cup S_{\infty} \cup T_0 \cup T_{\infty}$ we define $\mathcal{T}_{\infty}$ be the set of all the $J$-holomorphic planes of the broken leaves which intersect $S_{\infty}$. This set can be collectively compactified to obtain a smoothly embedded solid torus in $S^2 \times S^2$ whose boundary is $L$.  Similarly, the set $\mathcal{T}_{0}$ consisting of the other planes of the broken leaves can  be used to obtain another solid torus with boundary on $L$. Note that, since the planes in $\mathcal{T}_0$ and $\mathcal{T}_{\infty}$ fit together to form spheres in the class $(0,1)$,  by positivity of intersection a $J$-holomorphic sphere $u \colon S^2 \to S^2 \times S^2 \smallsetminus L$ in a class of the form $(1,d)$ must either intersect the all the planes of  $\mathcal{T}_{\infty}$ once, or all the planes of $\mathcal{T}_0$ once.

\begin{example}\label{cliff}
For the Clifford torus $L_{1,1} \subset S^2 \times S^2$ and a $J$ adapted to the standard parameterization $\psi_{1,1}$ of $L_{1,1}$, we get a  foliation  $\mathcal{F}_{1,1}$ of $S^2 \times S^2 \smallsetminus L_{1,1}$ with leaves  in the class $(0,1)$.  The broken leaves of $\mathcal{F}_{1,1}$ comprise  two families of $J$-holomorphic planes with boundary on $L_{1,1}$:
$\frak s_0$ which consisting of planes intersecting $S_0$ and $\frak s_{\infty}$ which consists of the planes intersecting $S_{\infty}$.
\end{example}

\begin{remark}\label{both}
One also obtains, for the same $J$, an analogous foliation whose leaves represent the class $(1,0)$.
\end{remark}

The following result establishes Refinement 2. The proof is based on that of Corollary E in \cite{rgi}.

\begin{proposition}
\label{move2}
Suppose that $\mathbb{L}$ is a monotone Lagrangian torus in $(S^2 \times S^2, \pi_1^*\omega + \pi_2^*\omega)$ that is disjoint from $L_{1,1}$. Then there is a Hamiltonian diffeomorphism $\phi$ of $S^2 \times S^2$which displaces $\mathbb{L}$ from  $S_0 \cup S_{\infty} \cup T_0 \cup T_{\infty}$ and is supported away from $L_{1,1}$. Moreover, $\phi(\mathbb{L})$ is homotopic to $L_{1,1}$ in the complement of $T_0 \cup T_{\infty}$ and also in the complement of $S_0 \cup S_{\infty}$.
\end{proposition}

\begin{proof}
We start with an almost-complex structure $J_0$ on $S^2 \times S^2$ adapted to the standard parameterization $\psi_{1,1}$ of $L_{1,1}$. We also assume that $J_0$ is standard in a Weinstein neighborhood $\mathcal{U}(\mathbb{L})$ of $\mathbb{L}$ that is disjoint from $L_{1,1}$. Let $\mathcal{F}_0$ be the corresponding $J_0$-holomorphic foliation of $S^2 \times S^2 \smallsetminus L_{1,1}$ and let $p_0 \colon  S^2 \times S^2 \to S_{\infty}$ be the corresponding map. We may assume that the points  $p_0(T_0)$ and $p_0(T_{\infty})$ lie in different components of $S_{\infty} \smallsetminus p_0(L_{1,1})$. 

Deform $J_0$ to a family $J_t$ for $t \ge 0$ by stretching the neck, as in \cite{behwz},  along a sphere bundle in $\mathcal{U}(\mathbb{L})$ that is disjoint from $L_{1,1}$. This yields a family of foliations $\mathcal{F}_t$ of $S^2 \times S^2 \smallsetminus L_{1,1}$. Since the planes of the broken leaves of $\mathcal{F}_0$ have minimal area they persist under the deformation to yield the planes of the broken leaves of $\mathcal{F}_t$. This yields a family of maps $p_t \colon  S^2 \times S^2 \to S_{\infty}$. 

\begin{lemma} The sets $p_t(\mathbb{L})$ in $S_{\infty}$ converge in the Hausdorff topology to a circle $C_{\infty} \in S_{\infty}$ as $t \to \infty $. 
\end{lemma}

\begin{proof} Let $J_{\infty}$ be the limiting almost complex structure which is fully stretched along $\mathbb{L}$.
The circle $C_{\infty}$ is the intersection with $S_{\infty}$ of the broken leaves of the $J_{\infty}$ foliation which are asymptotic to $\mathbb{L}$. Now, $p_t(\mathbb{L})$ consists of the intersection with $S_{\infty}$ of $J_t$ holomorphic spheres which intersect $\mathbb{L}$. Hence a sequence of points $z_t \in p_t(\mathbb{L})$ corresponds to a sequence of $J_t$ holomorphic curves in the class $(0,1)$ which all intersect $\mathbb{L}$. Up to taking a subsequence, this sequence of curves converges to a broken curve asymptotic to $\mathbb{L}$ and hence the $z_t$ converge to a point in $C_{\infty}$.
\end{proof}

\begin{lemma} If we denote the projection with respect to the fully stretched almost-complex structure by $p_{\infty}$, then $C_{\infty} = p_{\infty}(\mathbb{L})$  is disjoint from $p_{\infty}(L_{1,1})$.
\end{lemma}

\begin{proof} This follows from the fact that the original planes of the broken leaves have area $1$ and so cannot degenerate further. Indeed, since $\mathbb{L}$ is monotone, any holomorphic curve asymptotic to $\mathbb{L}$ must have integral area, and in particular curves in the class $(0,1)$ cannot converge to buildings with more than two top level curves.
\end{proof}

It follows from the results above that there is an $N>0$ such that $p_t(L_{1,1})$ is disjoint from $C_{\infty}$ for all $t \geq N$. With this we can choose  two continuous curves $\gamma_0,\, \gamma_{\infty} \colon [0, \infty) \to S_{\infty}$ with the following properties:
\begin{itemize}
  \item $\gamma_0(0) = p_0(T_0)$,  and $\gamma_{\infty}(0) = p_0(T_{\infty})$.
  \item $\gamma_0(t)$ and $\gamma_{\infty}(t)$ are disjoint from $p_t(L_{1,1})$ for all $t \in [0,\infty).$
  \item For some $N>0$,  both $\gamma_0(t)$ and $\gamma_{\infty}(t)$ are disjoint from $C_{\infty}$, and $C_{\infty}$ is disjoint from $p_t(L_{1,1})$ for all $t \geq N$.
  \item $C_{\infty}$ separates $\gamma_0(N)$ and $\gamma_{\infty}(N)$ in $S_{\infty}$.
\end{itemize}

For each $t \in [0, \infty)$, both $p_t^{-1}(\gamma_0(t))$  and $p_t^{-1}(\gamma_{\infty}(t))$ are $J_t$-holomorphic spheres in the class $(0,1)$ disjoint from $L_{1,1}$. The family of spheres $$\{p_t^{-1}(\gamma_0(t))\}_{t \in [0,N]}$$ forms a symplectic isotopy which displaces $T_0$ from $\mathbb{L}$ in the complement of $L_{1,1}$. Similarly, the family of spheres $$\{p_t^{-1}(\gamma_{\infty}(t))\}_{t \in [0,N]}$$ forms a symplectic isotopy which displaces $T_{\infty}$ from $\mathbb{L}$ in the complement of $L_{1,1}$. Moreover, these isotopies can be generated by a single Hamiltonian flow on $S^2 \times S^2$ that fixes $L_{1,1}$. The inverse flow displaces $\mathbb{L}$ from $T_0 \cup T_{\infty}$. The final separation condition is enough to guarantee the homotopy condition in the theorem. By considering also the $J_t$ holomorphic foliation in the class $(1,0)$ (see Remark \ref{both}),  we can displace $\mathbb{L}$ from $S_0 \cup S_{\infty}$ too. After adjusting the isotopy of $S_0 \cup S_{\infty}$ we may assume that it fixes $T_0 \cup T_{\infty}$, see Corollary 3.7 of \cite{rgi}. Hence the inverse flow will not reintroduce intersections with $T_0$ or $T_{\infty}$.
\end{proof}

\subsection{Refinement 3: We may assume that $\mathbb{L}$ is homologically trivial in $(S^2 \times S^2) \smallsetminus(S_0 \cup S_{\infty} \cup T_0 \cup T_{\infty})$}\label{claim3} To see this, note that $(S^2 \times S^2) \smallsetminus(S_0 \cup S_{\infty} \cup T_0 \cup T_{\infty})$
can be identified with a subset of the cotangent bundle of $\mathbb{T}^2$ in which  $L_{1,1}$ is identified with the zero section. In this setting we can invoke the following.

\begin{theorem}\label{hom}(Theorem 7.1,  \cite{rgi})
A homologically nontrivial Lagrangian torus $L$ in $(T^*\mathbb{T}^2, d\lambda)$ is Hamiltonian isotopic to a constant section. In particular if $L$ is exact then it is Hamiltonian isotopic to the zero section.
\end{theorem}

If our monotone Lagrangian $\mathbb{L}$ was homologically nontrivial in $(S^2 \times S^2) \smallsetminus(S_0 \cup S_{\infty} \cup T_0 \cup T_{\infty})$ it would then follow from Theorem \ref{hom} and Section $2.3.B_4''$ of \cite{gr} that  $\mathbb{L} \cap L_{1,1} \neq \emptyset$, which would contradict our original assumption.

\subsection{A path to the proof of Theorem \ref{one}}  By the three Refinements established above, it suffices to show that the following assumption is false.
\medskip

\noindent{\bf Assumption 2.}
There is a monotone Lagrangian torus $\mathbb{L}$ in the set  $$Y= (S^2 \times S^2) \smallsetminus(S_0 \cup S_{\infty} \cup T_0 \cup T_{\infty})$$ which is disjoint from the Clifford torus $L_{1,1}$ and is homologically trivial in $Y$.

\medskip

\noindent{\bf A path to a contradiction.} To obtain a contradiction to Assumption 2, we will show, using a sequence of blow-ups, inflations and blow-downs, that it implies the existence of  two disjoint monotone Lagrangian tori in a new (monotone) copy of $S^2 \times S^2$ which are both Hamiltonian isotopic to the Clifford torus therein, and hence can not be disjoint. 

To perform the necessary sequence of blow-ups, inflations and blow-downs,  we must  first establish the existence of a special collection symplectic spheres and disks in our current model (see Proposition \ref{intcount} below). These spheres and discs must be well-placed with respect to  a holomorphic foliation of $S^2 \times S^2 \smallsetminus (\mathbb{L} \cup L_{1,1})$ which we introduce below in Section \ref{double}. They are obtained from special holomorphic buildings whose existence we establish in Section \ref{FG}. These existence results rely on the analysis of a general stretching scenario that is contained in Section \ref{stretch}.

\begin{remark}
To falsify Assumption 2, we must use it to a build and analyze a complicated set of secondary objects in order to derive a contradiction. The reader is asked to bear in mind that many of the results  established in the remainder of this section hold in a setting which will later be shown to be impossible. \end{remark}

\subsection{Straightened holomorphic foliations of $S^2 \times S^2 \smallsetminus (\mathbb{L} \cup L_{1,1})$, under Assumption 2}\label{double} 

Let $\mathbb{L}$ be a Lagrangian torus as in Assumption 2.  Here we describe the holomorphic foliations of $S^2 \times S^2\smallsetminus (\mathbb{L} \cup L_{1,1})$ that are implied by the existence of  $\mathbb{L}$. 

Let $\psi$ be a parameterization of $\mathbb{L}$ and $\psi_{1,1} $ be the standard parameterization of $L_{1,1}$. Consider a tame almost complex structure $J$ on $$(S^2 \times S^2 \smallsetminus (\mathbb{L} \cup L_{1,1}), \pi_1^*\omega + \pi_2^*\omega)$$ which is adapted to both $\psi$ and $\psi_{1,1}$. We will always make the following assumption.

\begin{enumerate}
  \item [(A1)] $J$ is equal to the standard split complex structure near $S_0$, $S_{\infty}$, $T_0$ and $T_{\infty}$. In particular, $T_0$ and $T_{\infty}$ are unbroken leaves of the foliation.
 \end{enumerate}


Let  $J_{\tau}$ be the family of almost complex structures  on $S^2 \times S^2$ that are determined by $J$  as in \S 2.5 of \cite{rgi}. Taking the limit of the Gromov foliations in the class $(0,1)$ with respect to the $J_{\tau}$, as $\tau \to \infty$, and arguing as in \cite{rgi}, we  get a $J$-holomorphic foliation $$\mathcal{F}= \mathcal{F}(\mathbb{L},L_{1,1}, \psi,\psi_{1,1}, J)$$ of $S^2 \times S^2\smallsetminus (\mathbb{L} \cup L_{1,1})$. Each leaf of $\mathcal{F}$ still intersects $S_{\infty}$ in exactly one point, but there are now  three types of leaves. The first are unbroken leaves consisting of a single closed $J$-holomorphic sphere in $S^2 \times S^2\smallsetminus (\mathbb{L} \cup L_{1,1})$ of class $(0,1)$.  The second type of leaves are broken and consist of a pair of  finite energy $J$-holomorphic planes in $S^2 \times S^2\smallsetminus (\mathbb{L} \cup L_{1,1})$ that are asymptotic to $L_{1,1}$ along the same embedded geodesic with opposite orientations. As in  Example \ref{cliff}, the collection of planes like this which intersect $S_{\infty}$ comprise a $1$-dimensional family, $\frak s_{\infty}$, and their companion planes comprise a family $\frak s_{0}$. The third class of leaves are also broken, but consist of a pair of  finite energy $J$-holomorphic planes in $S^2 \times S^2\smallsetminus (\mathbb{L} \cup L_{1,1})$ asymptotic to $\mathbb{L}$. They too have matching ends. The planes of these broken leaves which intersect $S_{\infty}$ comprise the family $\mathcal{T}_{\infty}$ and the others comprise the family $\mathcal{T}_{0}$, as in \S \ref{solid}.

Refinement 3 has the following consequence.

\begin{lemma}\label{tauzero} The planes in $\mathcal{T}_{\infty}$ intersect both $S_0$ and $S_{\infty}$. Equivalently, the planes in $\mathcal{T}_{0}$ are disjoint from $S_0 \cup S_{\infty}$.
\end{lemma}

\begin{proof}
We define a relative homology class $\Sigma \in H_2(S^2 \times S^2, (S_0 \cup S_{\infty} \cup T_0 \cup T_{\infty})$ by first choosing an embedded path $\gamma:[0,1] \to S_{\infty}$ with $\gamma(0) = T_0 \cap S_{\infty}$ and $\gamma(1) = T_{\infty} \cap S_{\infty}$. Then choose a family of embedded paths $\sigma_t$ in $p^{-1}(\gamma(t))$ from $S_{\infty}$ to $S_0$. The union of the $\sigma_t$ define $\Sigma$. We may assume that $\gamma$ intersects $p((\mathbb{L})$ in a single point $\gamma(t_0)$, and, arguing by contradiction, if  $\mathcal{T}_{0}$ happened to intersect $S_0$ then $\sigma_t$ would intersect $\mathbb{L}$, giving a nontrivial intersection $\Sigma \bullet \mathbb{L}$. This contradicts Refinement 3.
\end{proof}

Note that there are now two foliation classes,  $\beta_{\mathbb{L}}$ and $\beta_{L_{1,1}}$, determined by each of the two classes of broken leaves. The foliation $\mathcal{F}$ also defines a projection map 
$$
p \colon S^2 \times S^2 \to S_{\infty}.
$$ 
In this setting, $p(L_{1,1})$ and $p(\mathbb{L})$ are disjoint embedded circles in $S_{\infty}$, which by Proposition \ref{move2} are disjoint from $T_0 \cup T_{\infty}$ and are homotopic in the complement. Therefore, without loss of generality, there are disjoint closed disks $A_0 \subset S_{\infty}$ with boundary $p(\mathbb{L})$ and $A_{\infty} \subset S_{\infty}$ with boundary $p(L_{1,1})$, such that $p(T_0) \in A_0$ and $p(T_{\infty}) \in A_{\infty}$.  Denote the closed annulus defined by the closure of $S_{\infty}\smallsetminus (A_0 \cup A_{\infty})$ by $B$.

Let $(P_1, P_2, Q_1, Q_2)$ be  coordinates in the neighborhood  $\mathcal{U}({\mathbb{L}})$ of  $\mathbb{L}$ determined by  $\psi$, and let $(p_1, p_2, q_1, q_2)$ be  coordinates in the neighborhood  $\mathcal{U}(L_{1,1})$ of $L_{1,1}$ determined by $\psi_{1,1}$. As in Lemma \ref{straight}, where we had only one Lagrangian torus, we may assume that the leaves of $\mathcal{F}$ are straight in both $\mathcal{U}({\mathbb{L}})$ and $\mathcal{U}(L_{1,1})$. In particular, we may assume that the unbroken leaves of $\mathcal{F}$ that intersect $\mathcal{U}(\mathbb{L})$ do so along the annuli $\{P_1=\delta \neq 0, Q_1 = \theta, |P_2|<\epsilon\}$, the planes of  $\mathcal{T}_{\infty}$   intersect $\mathcal{U}(\mathbb{L})$ along the annuli $\{P_1=0, Q_1 = \theta, 0<P_2<\epsilon\}$, and the planes of $\mathcal{T}_0$ intersect $\mathcal{U}(\mathbb{L})$ along the annuli 
$\{P_1=0, Q_1 = \theta, -\epsilon<P_2<0\}$. Similarly, we may assume that the unbroken leaves of $\mathcal{F}$ that intersect $\mathcal{U}(L_{1,1})$ do so along the annuli $\{p_1=\delta \neq 0, q_1 = \theta, |p_2|<\epsilon\}$, the planes of   $\frak{s}_{\infty}$   intersect $\mathcal{U}(L_{1,1})$ along the annuli 
$\{p_1=0, q_1 = \theta, 0<p_2<\epsilon\}$, and  the planes of $\frak{s}_0$ intersect $\mathcal{U}(L_{1,1})$ along the annuli $\{p_1=0, q_1 = \theta, -\epsilon<p_2<0\}$.

The map $p$ can also be described simply in these Weinstein neighborhoods. In $\mathcal{U}(\mathbb{L})$, we may assume that the region $\{P_1<0\} \subset \mathcal{U}(\mathbb{L})$ is mapped by $p$ into the interior of $A_0$,  and $\{P_1>0\} \subset \mathcal{U}(\mathbb{L})$ is mapped by $p$ into the interior of $B$. Similarly, we may assume that in $\mathcal{U}(L_{1,1})$ the region $\{p_1>0\} \subset \mathcal{U} (L_{1,1})$ is mapped by $p$ into the interior of $A_{\infty}$ and $\{p_1<0\} \subset \mathcal{U}(L_{1,1})$ is mapped by $p$ into the interior of $B$.

Using some of the freedoms available in the choice of $\psi$ and $\psi_{1,1} $, we can add the following additional assumption.  
\begin{enumerate}
  \item [(A2)] The foliation class $\beta_{\mathbb{L}}$ is equal to $(0,-1)\in H_1^{\psi}(\mathbb{L};\ZZ)$, and the foliation class $\beta_{L_{1,1}}$ is equal to $(0,-1)\in H_1^{\psi_{1,1}}(L_{1,1};\ZZ)$.
 \end{enumerate}

\subsection{Stretching scenario for class $(1,d)$, under Assumption 2.}\label{stretch} Let $J_{\tau}$, for $\tau \geq 0$, be the family of almost complex structures on $S^2 \times S^2$ used in  Section \ref{double} to obtain the foliation $\mathcal{F}$. For a sequence $\tau_k \to \infty$, let $u_{k,d} \colon S^2 \to S^2 \times S^2$ be a sequence of $J_{\tau_k}$-holomorphic curves in the class $(1,d)$ that converges to a holomorphic building $\mathbf{F}_d$ as in \cite{behwz}. The  limit $\mathbf{F}_d$ consists of genus zero holomorphic curves in three levels. The {\em top level} curves map  to $S^2 \times S^2 \smallsetminus (\mathbb{L} \cup L_{1,1})$ and are $J$-holomorphic. The {\em middle level} curves map to one of two copies of $\RR \times S^*\mathbb{T}^2$, the symplectization of the unit cotangent bundle of the flat torus. These copies correspond to 
$\mathbb{L}$ and $L_{1,1}$ and the identifications are defined by the parameterizations 
$\psi$ and $\psi_{1,1}$. It follows from the definition of the family $J_{\tau}$ that these middle level curves are all $J_{\mathrm{cyl}}$-holomorphic  where $J_{\mathrm{cyl}}$ is a fixed cylindrical almost complex structure. Similarly, the {\em bottom level} curves of the limiting building map to one of two copies of $T^*\mathbb{T}^2$  and are $J_{\mathrm{std}}$-holomorphic  where  $J_{\mathrm{std}}$ is a standard complex structure.  

Each top level curve of $\mathbf{F}_d$ can be compactified to yield a map from a surface of genus zero with boundary to $(S^2 \times S^2,\mathbb{L} \cup L_{1,1})$. The components of the boundary correspond to the negative punctures of the curve.  They are mapped to the closed geodesics on $\mathbb{L}$  or $L_{1,1}$ underlying the Reeb orbits to which the corresponding puncture is asymptotic. The middle and bottom level curves can be compactified to yield maps to either $\mathbb{L}$ or $L_{1,1}$ with the same type of boundary conditions. These compactified maps can  all be glued
together to form a map $\bar{\mathbf{F}}_d\colon S^2  \to S^2 \times S^2$ in the class $(1,d)$. 

%

\medskip

\begin{definition}\label{sp}
A $J$-holomorphic curve $u$ in $S^2\times S^2 \smallsetminus (\mathbb{L} \cup L_{1,1})$ is said to be {\em essential (with respect to the foliation $\mathcal{F}$)} if the map $p \circ u$ is injective.
\end{definition}

\begin{definition}\label{ft}
 Let $u$ be a $J$-holomorphic curve in $S^2\times S^2 \smallsetminus (\mathbb{L} \cup L_{1,1})$. A puncture of $u$ is said to be of {\em foliation type with respect to} $\mathbb{L}$  $(L_{1,1})$ if it is asymptotic to a closed Reeb orbit which lies on the copy of $S^*\mathbb{T}^2$ that corresponds to $\mathbb{L}$  ($L_{1,1}$) and covers a closed geodesic in an integer multiple of the foliation class $\beta_{\mathbb{L}}$ ($\beta_{L_{1,1}}$). The puncture is of {\em positive (negative) foliation type} if  this integer is positive (negative).
\end{definition}

\begin{lemma}\label{ends}
Let $u$ be a $J$-holomorphic curve in $S^2\times S^2 \smallsetminus (\mathbb{L} \cup L_{1,1})$ with a puncture. Let $\{c_l\}$ be a sequence of circles in the domain of $u$ which lie in a standard neighborhood of the puncture, wind once around it, and converge to it in the Hausdorff topology. If the puncture is of foliation type with respect to  $\mathbb{L}$  $(L_{1,1})$, then the sets $p(u (c_l))$ converge to a point on $p(\mathbb{L})$ $(p(L_{1,1}))$. Moreover each $p(u (c_l))$ either maps into the point (in which case $u$ covers a plane in a broken leaf) or it winds nontrivially around the point. If the puncture is not of foliation type then the sets $p(u (c_l))$ converge to $p(\mathbb{L})$ $(p(L_{1,1}))$.
\end{lemma}

\begin{proof}
This follows from the exponential convergence theorem from \cite{hofa}.
\end{proof}

\begin{corollary}\label{possible ends}
If $u$ is an essential  $J$-holomorphic curve in $S^2\times S^2 \smallsetminus (\mathbb{L} \cup L_{1,1})$, then its punctures on $\mathbb{L}$ are either all of foliation type or none of them are, and similarly for the punctures on $L_{1,1}$.  If $u$ has no punctures of foliation type, then it is either a J-holomorphic plane or cylinder. If $u$ is a plane, then the closure of the image of $p \circ u$ is $A_0$ or $A_{\infty}$ or the closure of their complements in $S_{\infty}$. If $u$ is a cylinder, then the closure of the image of $p \circ u$ is $B$.
\end{corollary}

\begin{proof}
The previous lemma implies that if $u$ has punctures of both foliation type and not of foliation type on $\mathbb{L}$ or $L_{1,1}$ then $p  \circ u$ will not be injective.
\end{proof}

The following result can be proved in the same way as Lemma 6.2  in \cite{hl}.

\begin{lemma}\label{allsame}
Let $u$ be an essential curve whose punctures on $\mathbb{L}$ are all of foliation type. Then these punctures are either all positive or all negative (see Definition \ref{ft}). The same holds for the punctures on $L_{1,1}$.
\end{lemma}

%

Let  $u_{k,d}$ be a sequence  converging to $\mathbf{F}_d$  as in the {\bf stretching scenario for class $(1,d)$.} The fact that the curves $u_{k,d}$ must intersect each leaf of $\mathcal{F}$ exactly once, imposes several important restrictions on $\mathbf{ F}_d$ in relation to the foliation $\mathcal{F}$, and allows us to identify a handful of possible limit structures. 
 
 \begin{proposition}\label{lem1} Let $\mathbf{ F}_d$ be a limit as in the {\bf stretching scenario for class $(1,d)$}. 
Then the building $\mathbf{ F}_d$ is of one of the following types:

\bigskip

\noindent {\bf Type 0.} $\mathbf{ F}_d$ is a (possibly nodal) $J$-holomorphic sphere in $S^2 \times S^2 \smallsetminus (\mathbb{L} \cup L_{1,1})$ in the class $(1,d)$, where one (essential) sphere lies in the class $(1,j)$ for some $1 \le j \le d$ and any remaining curves are either spheres covering unbroken leaves of the foliation, or pairs of planes covering broken leaves of the foliation.

\bigskip

\noindent {\bf Type 1.} $\mathbf{ F}_d$ has a unique essential curve $u_d$. The punctures of  $u_d$ are all of foliation type, and along $\mathbb{L}$, and also $L_{1,1}$, are either all positive or all negative. The image of  $p\circ u_d$ is $S_{\infty}$ minus finitely many points on $p(\mathbb{L}) \cup p(L_{1,1})$. The other top level curves of $\mathbf{ F}_d$ either cover unbroken leaves of the foliation, or they are $J$-holomorphic planes covering one of the planes of a broken leaf of the foliation.


\bigskip

\noindent {\bf Type 2a.} $\mathbf{ F}_d$ has exactly two essential  curves, $u_{\mathbb{L}}$ and $\underline{u}$. The closures of the images of the maps $p \circ u_{\mathbb{L}}$ and $p \circ \underline{u}$  are $A_0$ and $B \cup A_{\infty}$, respectively.  Any punctures of $\underline{u}$ on $L_{1,1}$ are all of foliation type and are either all positive or all negative.
The other top level curves of $\mathbf{ F}_d$  cover (broken or unbroken) leaves of $\mathcal{F}$. 

\bigskip

\noindent {\bf Type 2b.} $\mathbf{ F}_d$ has exactly two essential curves, $\underline{u}$ and $u_{L_{1,1}}$. The closures of the images of the maps $p \circ \underline{u}$ and $p \circ u_{L_{1,1}}$  are $A_0 \cup B$ and $A_{\infty}$, respectively.  Any punctures of $\underline{u}$ on $\mathbb{L}$ are all of foliation type and are either all positive or all negative.
The other top level curves of $\mathbf{ F}_d$  cover (broken or unbroken) leaves of $\mathcal{F}$. 

\bigskip

\noindent {\bf Type 3.} $\mathbf{ F}_d$ has exactly three essential curves, $u_{\mathbb{L}}$, $\underline{u}$,  and $u_{L_{1,1}}$. The closures of the images of the maps  $u_{\mathbb{L}}$, $\underline{u}$,  and $u_{L_{1,1}}$ are  $A_0$, $B$ and $A_{\infty}$, respectively. The other top level curves of $\mathbf{ F}_d$  again cover (broken or unbroken) leaves of $\mathcal{F}$. 

\end{proposition}

\medskip
\noindent{\bf Proof of Proposition \ref{lem1}.} We begin with the following result which allows us to use essential curves to sort the limit structures.

\begin{lemma}\label{lem0}
Let $\mathbf{ F}_d$ be a limit as in the {\bf stretching scenario for class $(1,d)$}. If $u$ is a top level curve of $\mathbf{ F}_d$, then it is either essential or else the image of $p \circ u$ is a point. The essential curves have disjoint images under $p$, which are open sets, and these images include the complement of $p(\mathbb{L}) \cup p(L_{1,1})$.
\end{lemma}

\begin{proof}
Recall that the curves of $\mathbf{F}_d$  can be compactified and glued
together to form a map $\bar{\mathbf{F}}_d\colon S^2  \to S^2 \times S^2$ in the class $(1,d)$. The intersections of $\bar{\mathbf{F}}_d$ with an unbroken leaf $T$ of $\mathcal{F}$ all correspond to intersections of top level curves of $\mathbf{ F}_d$ with $T$. Since $(1,d) \bullet T = (1,d) \bullet (0,1) =1$, there can only be one such intersection point, by positivity of intersection. If $u$ is a top level curve such that the map $p \circ u$ is constant, then $u$ covers part of a broken leaf of our foliation and has intersection number $0$ with all unbroken leaves. Assume then that $u$ is a top level curve such that $p \circ u$ is nonconstant. By the discussion above, $u$ intersects any unbroken leaf $T$ either once or not at all, and if $p \circ u$ has any double points then they must lie in $p(\mathbb{L}) \cup p(L_{1,1})$. Positivity of intersection again implies that the nonconstant map $p \circ u$ is an open mapping and this implies  that  the double points of $p \circ u$ form an open set. We conclude that $u$ is essential. To see that the essential curves have disjoint images under $p$ we can apply the same argument to a union $u \cup v$. The intersection number also implies that all unbroken fibers intersect at least one essential curve.
\end{proof}

Lemma \ref{lem0} implies that there is  an essential curve $u$ of $\mathbf{ F}_d$ that  intersects $T_0$. The closure of the image of $p \circ u$ must contain $A_0$. By Corollary \ref{possible ends} the following cases are exhaustive.

\medskip

\noindent{\em Case 1: $u$ has no punctures.} In this case, $p  \circ u$ must be a bijection onto $S_\infty$. Hence,  $u$ is  a $J$-holomorphic sphere in a class of the form $(1,j)$ for $j$  in $[0,d]$. By Lemma \ref{lem0} all the other top level curves of $\mathbf{ F}_d$ must cover leaves of the foliation. 

The top level curves of $\mathbf{ F}_d$ which cover fibres fit together to form a possibly disconnected curve in the class $(0,d-j)$. If $j=d$ then $\mathbf{ F}_d$ consists only of the curve $u$. Either way, the building is of Type 0. 

\medskip

\noindent{\em Case 2:  $u$ has punctures and they are all of foliation type.} In this case we claim that $\mathbf{ F}_d$ is of Type 1. By Lemma \ref{ends}, the image of the map $p  \circ u$ includes points in each component of the complement of $p(\mathbb{L}) \cup p(L_{1,1})$, and so by Lemma \ref{lem0} we have that $p \circ u$ is a bijection onto  $S_\infty$ minus a finite set of points on $p(\mathbb{L}) \cup p(L_{1,1})$.  The other top level curves of $\mathbf{ F}_d$ must either cover unbroken leaves of $\mathcal{F}$ or they are $J$-holomorphic planes covering one of the planes of a broken leaf of $\mathcal{F}$. The statement about positivity or negativity of punctures is Lemma \ref{allsame}. 

\medskip

\noindent{\em Case 3: $u$ has one puncture not of foliation type.}  Since $u$ intersects the leaf $T_0$, the closure of the image of $p \circ u$ is either $A_0$ or $A_0 \cup B$. In either case, $u$ does not intersect $T_{\infty}$.

Suppose that the closure of the image of $p  \circ u$ is $A_0$. By Lemma \ref{lem0}, there is an essential curve $v$ of $\mathbf{ F}_d$ that  intersects $T_{\infty}$, and the images of  $p  \circ u$ and $p  \circ v$ cannot intersect. Hence the closure of the image of $p  \circ v$ is either  $A_{\infty}$ or $B \cup A_{\infty}$. In the first case, $\mathbf{ F}_d$ is of Type 3 with 
$u_{\mathbb{L}}=v$ and $u_{L_{1,1}}=u$, where the third curve, $\underline{u}$, exists by Lemma \ref{lem0}.
In the second case, $\mathbf{ F}_d$ is of Type 2a with $u^0_{\mathbb{L}}=u$ and $u^{\infty}_{\mathbb{L}} =v$. 

If, instead, the closure of the  image of $p  \circ u$ is $A_0 \cup B$, then a similar argument implies that $\mathbf{ F}_d$ is of Type 2b.
%
%
\medskip

\noindent This completes the proof of Proposition \ref{lem1}.

\medskip

\subsection{The existence of special buildings, under Assumption 2.}\label{FG}

In this section we will establish the existence of two special limits of the {\bf stretching scenario for class $(1,d)$} when $d$ is sufficiently large. The following result will be used to exploit the large $d$ limit.

\begin{lemma}\label{mono} There exists an $\epsilon >0$ such that $$\mathrm{area}(u) \ge \epsilon  u \bullet (S_0 \cup S_{\infty})$$ for all $J$-holomorphic curves $u$ in $S^2\times S^2 \smallsetminus (\mathbb{L} \cup L_{1,1})$.
\end{lemma}

\begin{proof}
Fix an open neighborhood of $S_{\infty}$ of the form $\mathcal{N}_{\epsilon}=S_{\infty} \times D^2(\epsilon)$ where $D^2(\epsilon)$ is the open disc of area $\epsilon$. We may assume that  the closure of $\mathcal{N}_{\epsilon}$  is disjoint from $\mathbb{L} \cup L_{1,1}$ and, by (A1), we may assume that $J$ restricts to $\mathcal{N}_{\epsilon}$ as the standard split complex structure. Let $\pi_2 \colon S_{\infty} \times D^2(\epsilon) \to D^2(\epsilon)$ be projection  and set $$u_{\epsilon,\infty} = u |_{u^{-1}(\mathcal{N}_{\epsilon})}.$$ By perturbing $\epsilon$ if needed we may assume that $u^{-1}(\mathcal{N}_{\epsilon})$ is a smooth manifold. We have 
$$
\mathrm{degree}(\pi_2 \circ u_{\epsilon,d}) = u \bullet S_{\infty}.
$$
This implies 
\begin{eqnarray*}
\mathrm{area}(u_{\epsilon,\infty}) & \geq &  \int_{u^{-1}(\mathcal{N}_{\epsilon})} u_{\epsilon,\infty}^*(\omega\oplus \omega) \\
{} & \geq & \int_{(\pi_2 \circ u_{\epsilon,\infty})^{-1}(D^2(\epsilon))}  (\pi_2 \circ u_{\epsilon,\infty})^*\omega \\
{} & = &  \left(\int_{D^2(\epsilon)} \omega \right) u \bullet S_{\infty}\\
{} & = & \epsilon u \bullet S_{\infty}.
\end{eqnarray*}
A similar calculation for $S_0$ gives the result.
\end{proof}

\begin{proposition}\label{existence2} For all sufficiently large $d$, there exists a limiting building $\mathbf{ F}$  as in the {\bf stretching scenario for class $(1,d)$}  such that $\mathbf{ F}$ is of Type 3. The building consists of its three essential  top level curves, $u_{\mathbb{L}}$, $\underline{u}$,  and $u_{L_{1,1}}$, together with $d-1$ planes in 
$\mathcal{T}_{0} \cup \mathcal{T}_{\infty}$ and $d$ planes in $\frak{s}_{0} \cup \frak{s}_{\infty}$.
%
\end{proposition}

\begin{proof}
Fix $d+1$ points on $L_{1,1}$ and $d$ points on $\mathbb{L}$.  Let $J_{\tau}$, for $\tau \geq 0$, be the family of almost complex structures on $S^2 \times S^2$ from Section \ref{double} and for a sequence $\tau_k \to \infty$, let $u_{k} \colon S^2 \to S^2 \times S^2$ be a convergent sequence of $J_{\tau_k}$-holomorphic curves in the class $(1,d)$ that pass through the $2d+1$ constraint points. Their  limit, $\mathbf{F}$, is the desired building.

To see this we first note that the point constraints already preclude the possibility that $\mathbf{ F}$ is of Type 0. If $F$ had Type $1$, the point constraints on $L_{1,1}$ would imply, by Lemma \ref{allsame}, that $\mathbf{ F}$ must contain $d+1$ planes either all in $\frak{s}_{0}$ or all in $\frak{s}_{\infty}$. This contradicts the fact that the our curve has intersection number $d$ 
with $S_0$ and $S_{\infty}$. The same argument precludes the possibility that $\mathbf{ F}$ has Type 2a.

It remains to show that $\mathbf{ F}$ does not have Type 2b. Assuming that $\mathbf{ F}$ has Type 2b, we will show that it must  include a collection of curves of total area equal to two, that intersect $S_0 \cup S_{\infty}$ $d$ times. If $d$ is sufficiently large, this contradicts 
Lemma \ref{mono} above.

\medskip

\noindent{\bf Claim 1.} If $\mathbf{ F}$ has Type 2b, then it includes at least $d$ planes in $\frak{s}_{0} \cup \frak{s}_{\infty}$.

\medskip

To see this consider the subbuilding $\mathbf{F}_{1,1}$ of $\mathbf{ F}$
consisting of its middle and bottom level curves mapping to the copies of $\RR \times S^*\mathbb{T}^2$  and $T^*\mathbb{T}^2$ that correspond to $L_{1,1}$. Since it is connected and has genus zero, it follows from Proposition 3.3. of \cite{hl} that 
$$
\mathrm{index}(\mathbf{F}_{1,1}) =2(s-1),
$$
where $s$ is the number of positive ends of $\mathbf{F}_{1,1}$. Since, $\mathbf{F}_{1,1}$ passes through the $d+1$ generic point constraints on $L_{1,1}$, and the Fredholm index in these manifolds is nondecreasing under multiple covers, we must also have 
$$
\mathrm{index}(\mathbf{F}_{1,1}) \geq 2(d+1).
$$
Hence, $\mathbf{F}_{1,1}$ has at least $d+2$ positive ends.  Under the assumption that $\mathbf{ F}$ has Type 2b,  two of these positive ends match with the two essential top level curves of $\mathbf{ F}$. This leaves at least $d$ positive ends  of $\mathbf{F}_{1,1}$ that match with top level curves of $\mathbf{ F}$ that cover planes in $\frak{s}_{0} \cup \frak{s}_{\infty}$.

\begin{remark}
The same argument implies that if $\mathbf{ F}$ has Type 3 then again it must include at least  $d$ planes in $\frak{s}_{0} \cup \frak{s}_{\infty}$.
\end{remark}

\medskip

\noindent{\bf Claim 2.} If $\mathbf{ F}$ has Type 2b, then it includes $d$ planes in  $\mathcal{T}_{0}$ and none in $\mathcal{T}_{\infty}$.

\medskip

By Lemma \ref{allsame} the $d$ constraint points on $\mathbb{L}$ imply that, if  $\mathbf{F}$ is of Type 2b, then it must contain $d$ planes either all in $\mathcal{T}_{0}$ or all in $\mathcal{T}_{\infty}$.  To show that these planes can not be in $\mathcal{T}_{\infty}$, we consider intersections with 
$S_0 \cup S_{\infty}$. Overall, the top level curves of $\mathbf{ F}$ must intersect $S_0 \cup S_{\infty}$ exactly $2d$ times. The planes of $\mathbf{ F}$ asymptotic to $L_{1,1}$ from Claim 1, account for at least $d$ of these intersections.

Since $\mathbb{L}$ is homologically trivial in $Y$, by Lemma \ref{tauzero}  each plane of $\mathcal{T}_{\infty}$ must intersect both $S_0$ and $S_{\infty}$, while the planes in $\mathcal{T}_{0}$ intersect neither of these spheres. If the $d$ planes of $\mathbf{F}$ asymptotic to $\mathbb{L}$ are in $\mathcal{T}_{\infty}$ then they would contribute another $2d$ intersections with $S_0 \cup S_{\infty}$. By positivity of intersection, this can not happen and so these planes must belong to $\mathcal{T}_{0}$ as claimed.

\medskip

To complete the argument, we now balance areas. The total area of all the curves in $\mathbf{F}$ is $2(d+1)$. If $\mathbf{ F}$ has Type 2b, then the planes from Claim 1 and Claim 2 have total area at least $2d$. 
It's essential curves  must then have total area equal to 2. Also, they must contribute the remaining $d$ intersections with $S_0 \cup S_{\infty}$. It follows from Lemma \ref{mono}, that this is impossible for all $d$ sufficiently large.  Hence $\mathbf{ F}$ can not be of Type 2b, and must instead be of Type 3. Arguing as above, it follows that in addition to  its three essential  top level curves, $\mathbf{ F}$ must then have $d$ planes in $\frak{s}_{0} \cup \frak{s}_{\infty}$ and $d-1$ planes in 
$\mathcal{T}_{0} \cup \mathcal{T}_{\infty}$.
\end{proof}

%

\begin{proposition}\label{existence3} For all sufficiently large $d$, there exists a limiting building $\mathbf{G}$  as in the {\bf stretching scenario for class $(1,d)$} such that $\mathbf{ G}$ 
Type 3. In addition to its three essential curves
it consists of  $d$ planes in 
$\mathcal{T}_{0} \cup \mathcal{T}_{\infty}$ and $d-1$ planes in $\frak{s}_{0} \cup \frak{s}_{\infty}$. 
%
%
\end{proposition}

\begin{proof}
Here we fix $d$ points on $L_{1,1}$ and $d+1$ points on $\mathbb{L}$, and for $J_{\tau}$
as in Proposition \ref{existence2} consider the limit, $\mathbf{ G}$, of a  convergent sequence of $J_{\tau_k}$-holomorphic spheres, for $\tau_k \to \infty$,  that represent  the class $(1,d)$ and pass through the $2d+1$ constraint points. The point constraints imply that $\mathbf{ G}$ is not of Type 0. 

If $\mathbf{ G}$ was of  Type 1, the point constraints would imply that $\mathbf{ G}$ includes at least $d$ planes in either $\frak{s}_{0}$ or $\frak{s}_{\infty}$, and at least $d+1$ planes in either $\mathcal{T}_{0}$ or $\mathcal{T}_{\infty}$. From this it follows that the essential curve of $\mathbf{ G}$ would have area $1$.
Recalling Lemma \ref{tauzero}, since $L$ is homologicially trivial, the planes  of $\mathcal{T}_{\infty}$ each intersect $S_0 \cup S_{\infty}$ twice. Arguing as in Claim 2 from the proof of Proposition \ref{existence2}, if the planes asymptotic to $\mathbb{L}$ lie in $\mathcal{T}_{\infty}$ then the broken planes will contribute a total of $d + 2(d+1)$ intersections with $S_0 \cup S_{\infty}$, a contradiction as there are only $2d$ such intersections. On the other hand, if these planes all lie in $\mathcal{T}_{0}$ then the essential curve must contribute $d$ intersections with $S_0 \cup S_{\infty}$. As this essential curve has area $1$ then contradicts Lemma \ref{mono} when $d$ is sufficiently large. Hence, $\mathbf{ G}$ is not of Type 1.

Next we show that $\mathbf{ G}$  can not be of Type 2b. Assume that it is. Then $\mathbf{ G}$ includes $d+1$ planes in either $\mathcal{T}_{0}$ or $\mathcal{T}_{\infty}$.  Counting intersections as above, $\mathbf{ G}$ must have $d+1$ planes in $\mathcal{T}_{0}$.

Arguing as in Claim 1 above, we consider the subbuilding $\mathbf{G}_{1,1}$ of $\mathbf{G}$
consisting of its middle and bottom level curves that map to the copies of $\RR \times S^*\mathbb{T}^2$  and $T^*\mathbb{T}^2$ that correspond to $L_{1,1}$. Since $\mathbf{G}_{1,1}$ is connected and has genus zero, we  have
$$
\mathrm{index}(\mathbf{G}_{1,1}) =2(s-1),
$$
where $s$ is the number of positive ends of $\mathbf{G}_{1,1}$. Since, $\mathbf{G}_{1,1}$ passes through the $d$ generic point constraints on $\mathbb{L}$ we also have 
$$
\mathrm{index}(\mathbf{G}_{1,1}) \geq 2d.
$$
Hence,  $\mathbf{G}_{1,1}$ has at least $d+1$ positive ends.  Two of these positive ends match with negative ends of the two essential curves of $\mathbf{G}_{1,1}$. It follows that $\mathbf{G}$ 
must have at least $d-1$ planes in  $\frak{s}_{0} \cup \frak{s}_{\infty}$. This means the planes covering broken leaves then have area at least $2d$. As the limiting building has total area $2d+2$ and also includes two essential curves we see that the essential curves each have area $1$ and there are exactly $d-1$ planes in  $\frak{s}_{0} \cup \frak{s}_{\infty}$.
As the planes in $\mathcal{T}_{0}$ are disjoint from $S_0 \cup S_{\infty}$, the essential curves of $\mathbf{ G}$ must have $d+1$ intersections with $S_0 \cup S_{\infty}$. Lemma \ref{mono} again implies that this is impossible for all sufficiently large $d$.  

Finally we show that $\mathbf{ G}$  can not be of Type 2a. In this case $\mathbf{ G}$ includes $d$ planes in $\mathcal{T}_{0} \cup \mathcal{T}_{\infty}$ and $d$ planes in either $\frak{s}_{0}$ or  $\frak{s}_{\infty}$. The planes asymptotic to $L_{1,1}$ thus account for all intersections with either $S_0$ or $S_{\infty}$ and so the planes asymptotic to $\mathbb{L}$ therefore all lie in $\mathcal{T}_{0}$. The essential curves have total area $2$ and must together account for all intersections with either $S_0$ or $S_{\infty}$. This contradicts Lemma \ref{mono} as before.
\end{proof}

\begin{lemma}\label{area1} All curves in the limiting buildings $\mathbf{ F}$ and $\mathbf{ G}$ that map to $S^2 \times S^2 \smallsetminus (\mathbb{L} \cup L_{1,1} )$ have area $1$, and in particular are simply covered. 
\end{lemma}

\begin{proof}
To see this, note that since $\mathbf{ F}$ is of Type 3, it has its three essential curves together with $2d-1$ other top level curves that cover leaves of the foliation.  Since $\mathbf{ F}$ has total area $2d+2$ and monotonicity implies that all curves have integral area, the result for $\mathbf{ F}$ follows.

The same argument applies to $\mathbf{ G}$.
\end{proof}


\subsection{A collection of symplectic spheres and disks, under Assumption 2.}

 Consider  $(S^2 \times S^2, \pi_1^*\omega + \pi_2^*\omega)$ equipped with an almost complex structure $J$ adapted to  parameterizations $\psi$ and $\psi_{1,1}$ of $\mathbb{L}$ and $L_{1,1}$, respectively. 
%
Recall that for the projection $p : S^2 \times S^2 \to S_{\infty}$, defined by the foliation  $\mathcal{F}$ corresponding to $J$, the images $p(\mathbb{L})$ and $p(L_{1,1})$ are disjoint circles. There are also disjoint disks $A_0 \subset S_{\infty}$ with boundary $p(\mathbb{L})$ and $A_{\infty} \subset S_{\infty}$ with boundary $p(L_{1,1})$ such that $p(T_0) \in A_0$ and $p(T_{\infty}) \in A_{\infty}$. In this section we will prove the following result.

 \begin{proposition}\label{intcount} There exist  embedded  symplectic spheres $F, G \colon S^2 \to S^2 \times S^2$ in the class $(1,d)$, and embedded  symplectic disks $\mathbb{E} \colon (D^2,S^1) \to (S^2 \times S^2,\mathbb{L})$ and $E_{1,1} \colon (D^2,S^1) \to (S^2 \times S^2,L_{1,1})$ of Maslov index $2$,  such that:
 
\begin{enumerate}
\item  $F$, $G$, $\mathbb{E}$ and $E_{1,1}$ are all $J$-holomorphic away from arbitrarily small neighborhoods of a collection of Lagranian tori whose elements are near to, and Lagrangian isotopic to, either $\mathbb{L}$ or $L_{1,1}$;\\
\item the class of $\,\mathbb{E}|_{S^1}$ and the foliation class $\beta_{\mathbb{L}}$  form an integral basis of $H_1(\mathbb{L}:\ZZ)$;\\
\item the class of $E_{1,1}|_{S^1}$ and the foliation class $\beta_{L_{1,1}}$ form an integral basis of $H_1(L_{1,1}:\ZZ)$;\\
\item exactly one of $F$ and $G$ intersects  the planes of  $\mathcal{T}_{0}$ and the other intersects  the planes of $\mathcal{T}_{\infty}$;\\
\item exactly one of $F$ and $G$ intersects  the planes of $\frak{s}_{0}$ and the other intersects the planes of $\frak{s}_{\infty}$;\\
\item $F \bullet \mathbb{E} + G \bullet \mathbb{E} =d$;\\
\item $F \bullet E_{1,1} + G \bullet E_{1,1} =d$;\\
\item $F \bullet G =2d$;\\
\item $p(F \cap G)$ consists of $d$ points in $A_0$ and $d$ points in $A_{\infty}$.\\
\end{enumerate} 
\end{proposition}

\noindent{\bf Proof of Proposition \ref{intcount}.}
To prove this result we will compactify and deform curves of the buildings $\mathbf{ F}$ and $\mathbf{ G}$ from Propositions \ref{existence2} and \ref{existence3}. There are two deformation processes which are used in order to resolve the intersection patterns of the resulting maps.  The fact that the foliation $\mathcal{F}$ is assumed to be straightened near $\mathbb{L}$ and $L_{1,1}$ (see Section \ref{double}) plays a prominent role here.  In particular, curves through $\mathcal{U}(\mathbb{L})$ and $\mathcal{U}(L_{1,1})$ are deformed  so that many of their intersections can be identified with intersections between cylinders in the local models.

\subsubsection{Deformations near $\mathbb{L}$} We begin by describing two deformation processes for curves with ends on $\mathbb{L}$. These appear as Lemma \ref{fukaya} and Corollary \ref{up}, below.

Consider the coordinates $(P_1,Q_1,P_2,Q_2)$ in our Weinstein neighborhood $$
\mathcal{U}(\mathbb{L}) = \{|P_1|<\epsilon,\, |P_2|<\epsilon\}.
$$
For each  {\it translation vector} $\mathbf{ v}=(a,b) \in (-\epsilon, \epsilon)^2$, there is a corresponding nearby Lagrangian torus
$$\mathbb{L}(\mathbf{ v}) = \mathbb{L}(a,b) = \{P_1=a,P_2=b\} \subset \mathcal{U}(\mathbb{L}).$$ Note that the parameterization $\psi$ of $\mathbb{L}$ determines an obvious parameterization, $\psi(\mathbf{ v}) =\psi(a,b)$ of $\mathbb{L}(a,b)$, and a canonical isomorphism from
$H_1^{\psi}(L; \ZZ)$ to $H_1^{\psi(a,b)}(L(a,b); \ZZ)$.

Following \cite{rgi} section 4,
given a finite collection of translation vectors,
$$
\mathbf{V} =\{\mathbf{v}_1, \dots, \mathbf{v}_k \} =\{(a_1,b_1), \dots, (a_k,b_k)\},
$$
let  $J_{\mathbf{V} }$  be an almost complex structure which coincides with $J$ outside $\mathcal{U}(\mathbb{L})$ and inside has the form
\begin{equation}
\label{J|}
J_{\mathbf{V} } \frac{\partial}{\partial Q_i} =  - \rho_{\mathbf{V} } \frac{\partial}{\partial P_i},
\end{equation}
where $\rho_{\mathbf{V} }$ is a positive function away from the collection of Lagrangians $$ \mathbb{L}(\mathbf{V} )= \cup_{i=1}^k \mathbb{L}(\mathbf{v}_i),$$   and in a neighborhood of each $\mathbb{L}(\mathbf{v}_i)$  has the form $$\rho_{\mathbf{V} } = \sqrt{(P_1-a_i)^2 + (P_2-b_i)^2}.$$ In this case, we say that $J_{\mathbf{V} }$ is stretched along  $\mathbb{L}({\mathbf{V} })$. The set of all such almost complex structures will be denoted by $\mathcal{J}_{\mathcal{U}(\mathbb{L})}$.

Using the induced parameterizations  $\psi(\mathbf{v}_i)$ of each of the $\mathbb{L}(\mathbf{v}_i)$, the almost complex structure $J_{\mathbf{V} }$ defines a family of almost complex structures $J_{\mathbf{V} ,\tau}$ on $S^2 \times S^2$ which allow us to stretch $J$ holomorphic curves in $S^2 \times S^2$ along $\mathbb{L}(\mathbf{V} ) \cup L_{1,1}.$  The limit of the Gromov foliations for the $J_{\mathbf{V} ,\tau}$, in class $(0,1)$, yield a foliation $\mathcal{F}(\mathbf{V} )$ of $$S^2 \times S^2 \smallsetminus (\mathbb{L}(\mathbf{V} ) \cup L_{1,1}).$$
For example, for $\mathbf{V} =\{(0,0)\}$ we have $J_{\mathbf{V} } =J$ and $\mathcal{F}(\mathbf{V} )=\mathcal{F}$.
%

\begin{lemma}\label{folcor} Leaves of the foliation $\mathcal{F}(\mathbf{V} )$ intersect $\mathcal{U}_{\epsilon}(\mathbb{L})$ along the annuli $\{P_1=\delta, Q_1 = \theta, |P_2|< \epsilon \}$. A leaf of $\mathcal{F}(\mathbf{V} )$ that intersects $\mathcal{U}_{\epsilon}(\mathbb{L})$ along  the annulus $\{P_1=\delta, Q_1 = \theta, |P_2|< \epsilon \}$ is broken if and only if the collection $\mathbf{V} $ contains an element of the form  $(\delta,b_i)$.
\end{lemma}

\begin{proof} It follows from \eqref{J|} that these annuli are $J_{\mathbf{V} }$-holomorphic. By assuming $J$ satisfies the conclusions of Lemma \ref{straight}, they also extend to $J_{\mathbf{V} }$-holomorphic spheres in the class $(0,1)$. By positivity of intersection, these spheres, and indeed any holomorphic sphere in the class $(0,1)$, are leaves of the foliation $\mathcal{F}(\mathbf{V} )$.
\end{proof}

Our first deformation process allows us to deform a regular curve so that its ends on $\mathbb{L}$ become ends on a nearby translated Lagrangian.

 \begin{lemma}\label{fukaya}(Fukaya's Trick) Let $u$ be a regular $J$-holomorphic curve with $k\geq0$ ends on $\mathbb{L}$ and $l\geq 0$ ends on $L_{1,1}$.
For all $\mathbf{v}=(a,b)$ with  $\|\mathbf{v}\|^2= a^2 + b^2$ sufficiently small there is a regular $J_{\mathbf{v}}$-holomorphic curve $u(\mathbf{v})$ with $k$ ends on $\mathbb{L}(\mathbf{v})$ and $l$ ends on $L_{1,1}$. Moreover the ends of  $u(\mathbf{v})$  on $\mathbb{L}(\mathbf{v})$ represent the identical classes in $H_1^{\psi(\mathbf{v})}(L,\RR)$ as do those of $u$ in $H_1^{\psi}(L,\RR)$. The classes corresponding to the ends of $u(\mathbf{v})$ on $L_{1,1}$ are also identical to those of $u$.
\end{lemma}

\begin{proof}

For $\|\mathbf{ v}\|$ sufficiently small, the Lagrangian isotopy $t \mapsto \mathbb{L}(t\mathbf{v})$ for $0\le t \le 1$ is contained in $\mathcal{U}(\mathbb{L})$. Let $f_{t,\mathbf{v}}$ be a family of diffeomorphisms of $S^2 \times S^2$ such that:
 \begin{itemize}
  \item $f_{0,\mathbf{v}}$ is the identity map,
  \item $f_{t,\mathbf{v}}(\mathbb{L}) = \mathbb{L}(t\mathbf{v})$ for all $t \in [0,1]$,
  \item each $f_{t,\mathbf{v}}$ is equal to the identity map outside of $\mathcal{U}(\mathbb{L})$,
  \item $\|f_{t,\mathbf{v}}\|_{C^1}$ is of order 1 in $\|\mathbf{v}\|$. 
\end{itemize}

As above, the parameterization $\psi$
of $\mathbb{L}$ determines a parameterization $\psi(t\mathbf{v})$ for each $\mathbb{L}(t\mathbf{v})$. Let $J_{t\mathbf{v}}$ be a family of tame almost structures in $\mathcal{J}_{\mathcal{U}(\mathbb{L})}$ such that each $J_{t\mathbf{v}}$ is adapted to $\psi(t\mathbf{v})$. In particular, $J_{t\mathbf{v}}$ is stretched along $\mathbb{L}(t\mathbf{v})$. Set $$\tilde{J}_{t\mathbf{v}} = (f^{-1}_{t,\mathbf{v}})_* J_{t\mathbf{v}}.$$ For $\|\mathbf{v}\|$ sufficiently small, $\tilde{J}_{t\mathbf{v}}$ is a tame almost complex structure on $S^2 \times S^2 \smallsetminus \mathbb{L }\cup L_{1,1}$ for all $t \in [0,1]$. Since $u$ is regular,  for sufficiently small $\|\mathbf{v}\|$ the curve $u$ persists to yield a regular $\tilde{J}_{\mathbf{v}}$-holomorphic curve $\tilde{u}( \mathbf{v})$ with the same asymptotic behavior as $u$.  By our choice of $\tilde{J}_{t\mathbf{v}}$, $$ u(\mathbf{v})=f_{1,\mathbf{v}}\circ \tilde{u}(\mathbf{v})$$ is then a regular $J_{\mathbf{v}}$-holomorphic curve with $k$  ends on $\mathbb{L}(\mathbf{v})$ instead of $\mathbb{L}$. 

\end{proof}

By Lemma \ref{area1}, one can apply Lemma \ref{fukaya} to all the top level curves curves of the buildings $\mathbf{ F}$ and $\mathbf{ G}$ from Proposition \ref{existence2} and Proposition \ref{existence3},
to obtain new buildings $\mathbf{ F}(\mathbf{v})$ and $\mathbf{G}(\mathbf{v})$. Indeed, Lemma \ref{area1} implies that the top level curves are somewhere injective (and are actually embedded as they are limits of embedded curves) so for a generic choice of $J$ they are regular. Applying Theorem 1 from \cite{we}, we even have the stronger  statement that our curves are regular for all $J$. 

For example, suppose that the top level curves of $\mathbf{ F}$ are 
\begin{equation*}
\label{ }
\{u_{\mathbb{L}},\underline{u},u_{L_{1,1}}, u_1, \dots,  u_{d-1},\frak{u}_1, \dots, \frak{u}_d\}, 
\end{equation*}
where the 
$u_i$ belong to $\mathcal{T}_0 \cup \mathcal{T}_{\infty}$ and the $\frak{u}_j$ belong to $\frak{s}_0 \cup \frak{s}_{\infty}$. Then for $\mathbf{ v}=(a,b)$ with $\|\mathbf{v}\|$ sufficiently small we can define the deformed building $\mathbf{ F}(\mathbf{v})$ to be the one whose top level curves are
 \begin{equation*}
\label{ }
\{ u_{\mathbb{L}}(\mathbf{v}), \underline{u}(\mathbf{v}),u_{L_{1,1}},u_1(\mathbf{v}), \dots, u_{d-1}(\mathbf{v}),\frak{u}_1, \dots, \frak{u}_d\}
\end{equation*}
and whose middle and bottom level curves are the same as those of $\mathbf{ F}$ but are now 
considered to map to copies of $\RR \times S^* \mathbb{T}^2$ and $T^* \mathbb{T}^2$ that correspond to $\mathbb{L}(\mathbf{ v})$ rather than $\mathbb{L}$. Note that 
$\mathbf{ F}(\mathbf{v})$ still has a continuous compactification $\bar{\mathbf{ F}}(\mathbf{v}) \colon S^2 \to S^2 \times S^2$ which can be deformed arbitrarily close to $\mathbb{L}(\mathbf{ v})$ to obtain a smooth symplectic sphere $F=F(\mathbf{ v}) \colon S^2 \to S^2 \times S^2$ which is $J$-holomorphic away from a small neighborhood of $\mathbb{L}(\mathbf{ v})$.

Our second deformation process changes the essential $J$-holomorphic curve $u_{\mathbb{L}}$ of $\mathbf{ F}$ 
into one which is pseudo-holomorphic with respect to an almost-complex structure that is stretched along additional nearby Lagrangian tori.

\begin{lemma}\label{away}
For $b \neq 0$, set $\mathbf{V} = \{ (0,0), (0,b)\}$. Let $J_s$, for $s \in [0,1]$, be a smooth family of almost complex structures in $\mathcal{J}_{\mathcal{U}(\mathbb{L})}$ that connects $J$ to $J_{\mathbf{V} }$ in a manner that manifests the stretching of $J$ along $\mathbb{L}((0,b))$. 
The essential curve  $u_{\mathbb{L}}$ of $\mathbf{F}$ belongs to a smooth family  of  $J_s$-holomorphic planes $u_{\mathbb{L}}(s)$ for $s \in [0,1]$.
\end{lemma}

\begin{proof} By Lemma \ref{area1}, the initial curve $u_{\mathbb{L}}$ has area equal to $1$. Since $\mathbb{L}$ is monotone, no degenerations are possible until $s=1$. In other words, the family of deformed curves $u_{\mathbb{L}}(s)$  exists  for all $s \in [0,1)$ and it suffices to show that it extends to $s=1$.
Arguing by contradiction, assume that there is a sequence $s_j \to 1$ such that the curves $u_{\mathbb{L}}(s_j)$ converge to a nontrivial $J_{\mathbf{V} }$-holomorphic building $\mathbf{ H}$ which includes curves with punctures  asymptotic to $\mathbb{L}(\mathbf{v})$ with $\mathbf{v}=(0,b)$. We will show that this implies that, unlike $u_{\mathbb{L}}$, none of the curves of $\mathbf{H}$ intersect $T_0$, a contradiction.

\medskip
\noindent{Claim 1.}  Let $v$ be a $J_{\mathbf{V} }$-holomorphic curve of $\mathbf{ H}$. Any puncture of $v$ asymptotic to $\mathbb{L}(\mathbf{v})$ must cover a closed geodesic in a class $(k,l) \in  H_1(\mathbb{L}(\mathbf{v});\ZZ)$ with $k \leq0$.  
\medskip

\begin{proof} Since the closure of $p \circ u_{\mathbb{L}}$ is $A_0$, $u_{\mathbb{L}}$  is disjoint from the leaves of $\mathcal{F}$ which intersect $\mathcal{U}(\mathbb{L})$ in the region $\{P_1 >0\}$. The same is true  of the curves  $u_{\mathbb{L}}(s)$ for all $s<1$. Hence, $v$ must also be disjoint from these leaves. The curve $v$ can be extended smoothly to the oriented blow-up of the relevant puncture, such that the resulting map $\bar{v}$  acts on the corresponding boundary circle as $$\theta \mapsto (0,b, Q_1 + k \theta, Q_2 + l \theta)$$ for some $Q_1, Q_2 \in S^1$. The tangent space to the image of $\bar{v}$ at a boundary point on the circle  is spanned by $\{ k \frac{\partial}{\partial Q_1} + l \frac{\partial}{\partial Q_2}, k \frac{\partial}{\partial P_1} + l \frac{\partial}{\partial P_2} \}$. If $k$ were positive, this would contradict the fact that $v$ is disjoint from the leaves through $\{P_1 >0\}$ since $\mathbf{v}=(0,b)$.\end{proof}

\medskip
\noindent{Claim 2.}  Let $v$ be a $J_{\mathbf{V} }$-holomorphic curve with a  puncture  that is asymptotic to $\mathbb{L}(\mathbf{v})$ along a geodesic in a class which is a multiple of the foliation class, i.e. of the form $(0,l) \in  H_1^{\psi(\mathbf{v})}(\mathbb{L}(\mathbf{v});\ZZ)$.  Then $v$ must cover a plane or cylinder of a twice broken leaf of the foliation $\mathcal{F}(\mathbf{V} )$.
\medskip

\begin{proof}
This follows from the asymptotic properties of holomorphic curves and the fact that $v$ lies in $\{P_1 \le 0\}$, as in Lemma 6.2 of \cite{hl}.
\end{proof}

We can now complete the proof of Lemma \ref{away}.
Let $\mathbf{ H}_{\mathrm{ top}}$ denote the collection of top level curves of $\mathbf{ H}$, let
$\mathbf{H}_{1,1}$ be the subbuilding consisting of the middle and bottom level curves of $\mathbf{ H}$ that map to the copies of $\RR \times S^*\mathbb{T}^2$  and $T^*\mathbb{T}^2$ corresponding  to $L_{1,1}$, and let
$\mathbf{H}_{\mathbf{v}}$ be the subbuilding consisting of the middle and bottom level curves of $\mathbf{ H}$ that map to the copies of $\RR \times S^*\mathbb{T}^2$  and $T^*\mathbb{T}^2$ corresponding  to $\mathbb{L}(\mathbf{v})$.

Now consider the classes  $(k_1,l_1), \dots,(k_m,l_m) \in H_1(\mathbb{L}(\mathbf{v});\ZZ)$ of the geodesics determined by all of the punctures of top level curves of $\mathbf{ H}$ that are asymptotic to  $\mathbb{L}(\mathbf{v})$. These constitute the boundary of the cycle in $\mathbb{L}(\mathbf{v})$ that is obtained by gluing together the compactifications of the curves of $\mathbf{H}_{\mathbf{v}}$. Hence, the sum of the classes $(k_1,l_1), \dots,(k_m,l_m)$ must be $(0,0)$ and, by  Claim 1, each $k_i$ must be zero. It then follows from Claim 2, that any curve of $\mathbf{ H}$ with an end on $\mathbb{L}(\mathbf{v})$ must cover a plane or cylinder of a broken leaf of $\mathcal{F}(\mathbf{v})$.  

Now partition the curves of $\mathbf{ H}_{\mathrm{ top}} \cup \mathbf{H}_{\mathbf{v}} =\mathbf{ H} \smallsetminus \mathbf{ H}_{1,1}$ into connected components based on the matching of their ends in the copies of $\RR \times S^*\mathbb{T}^2$  and $T^*\mathbb{T}^2$ corresponding  to $\mathbb{L}({\mathbf{v}})$. Denote these components by $\mathbf{ H}_1, \dots, \mathbf{ H}_k$. The compactification of each $\mathbf{ H}_j$ is a cycle representing a  class  in $\pi_2(S^2 \times S^2,L_{1,1})$. By monotonicity, the symplectic area of this cycle is a positive integer. Since the area of $u_{\mathbb{L}}$ is one, we must have $k=1$ and the area of the cycle determined by $\mathbf{ H}_1$ must be one. By our assumption, $\mathbf{ H}_1$ must contain a curve with an end on $\mathbb{L}(\mathbf{v})$. By the discussion above, this implies that all the curves of $\mathbf{ H}_1$ must cover a plane or cylinder of a broken leaf of $\mathcal{F}(\mathbf{V} )$ through $\mathbb{L}(\mathbf{v})$. None of these leaves intersect $T_0$ , and  neither do the curves of $\mathbf{ H}_{1,1}$. Hence, no curve of $\mathbf{ H} = H_1 \cup H_{1,1}$ intersects $T_0$, which is the desired contradiction.
\end{proof}

The $J_{\mathbf{V} }$-holomorphic curve $u_{\mathbb{L}}(1)$ is disjoint from the region $\{P_1>0\}$. By positivity of intersection, since it does not cover a leaf of the foliation, $u_{\mathbb{L}}(1)$ is disjoint from the hypersurface $\{P_1=0\}$. The closure of $p\circ u_{\mathbb{L}}(1)$ is equal to  $A_0$ and $u_{\mathbb{L}}(1)$ intersects the leaves of $\mathcal{F}(\mathbf{V} )$, that pass through the annuli $\{P_1 =c<0, Q_1=\theta\}$, exactly once. 
Arguing as above, one sees that the statement of Lemma \ref{away} also holds for  
$$\mathbf{V} =((0,0),(0,b_1),(0,b_2)),$$ for any  nonzero $b_1, b_2$ in $(-\epsilon, \epsilon)$.

Translating these Lagrangian tori slightly in the $P_1$-direction, we then get the following, more general, version of our second deformation process.

\begin{corollary}
\label{up}
 Let  $u_{\mathbb{L}}$ be the essential curve of $\mathbf{ F}$ which  is mapped by $p$ onto $A_0$. Choose nonzero constants $b_1, b_2$ in $(-\epsilon, \epsilon)$. If $\delta>0$ is sufficiently small, then for any $a_1, a_2$ in $(-\delta, \delta)$ and $$\mathbf{V} =\{(0,0),(a_1,b_1),(a_2,b_2)\} $$ there is a  $J_{\mathbf{V} }$-holomorphic curve $$u_{\mathbb{L}}^{\mathbf{V} }\colon \CC \to S^2 \times S^2 \smallsetminus (L(\mathbf{V} ) \cup L_{1,1})$$ in the class of $u_{\mathbb{L}}$ such that $u_{\mathbb{L}}^{\mathbf{V} }$ is disjoint from the region $\{P_1>0\}$,  the closure of the image of $p\circ u_{\mathbb{L}}^{\mathbf{V} }$ is $ A_0$, and $u_{\mathbb{L}}^{\mathbf{V} } $ intersects the leaves of $\mathcal{F}(\mathbf{V} )$, that pass through the annuli $\{P_1 =c<0, Q_1=\theta\}$, exactly once.
\end{corollary}

\subsubsection{Intersections near $\mathbb{L}$} We now use the deformation tools of  Lemma \ref{fukaya} and Corollary \ref{up} to  resolve some  intersection patterns. Let $\mathbf{ F}$ be a building of Type 3 as in  Proposition \ref{existence2} and consider translation data $$\mathbf{V} =\{\mathbf{ 0},\mathbf{ v}_1, \mathbf{ v}_2\} = \{ (0,0),(a_1,b_1), (a_2,b_2)\}.$$ In what follows we will always assume that $\mathbf{ v}_1$ and $\mathbf{ v}_2$ are distinct and nontrivial.
The collection of top level curves of $\mathbf{ F}$ is of the form 
\begin{equation*}
\label{ }
\{u_{\mathbb{L}},\underline{u},u_{L_{1,1}}, u_1, \dots,  u_{d-1},\frak{u}_1, \dots, \frak{u}_d\},
\end{equation*}
where the 
$u_1, \dots, u_{\alpha_0}$ belong to $\mathcal{T}_0$, $u_{\alpha_0+1}, \dots, u_{d-1}$ belong to $\mathcal{T}_{\infty}$, and the $\frak{u}_j$ belong to $\frak{s}_0 \cup \frak{s}_{\infty}$. If $\|\mathbf{ v}_1\|$ is sufficiently small then, as described in Remark \ref{defF},  the deformed building $\mathbf{ F}(\mathbf{v}_1)$ is well-defined and its top level curves are
\begin{equation*}
\label{ }
\{ u_{\mathbb{L}}(\mathbf{v}_1), \underline{u}(\mathbf{v}_1),u_{L_{1,1}},u_1(\mathbf{v}_1), \dots, u_{d-1}(\mathbf{v}_1),\frak{u}_1, \dots, \frak{u}_d\}.
\end{equation*}
Choosing $a_1$ to be smaller still, if necessary, we may assume that  Corollary \ref{up} holds for $\mathbf{V}$ for $|a_2|$ sufficiently small. This yields  a
$J_{\mathbf{V} }$-holomorphic curve $u_{\mathbb{L}}^{\mathbf{V} }$ which is disjoint from the region $\{P_1>0\}$ and intersects the leaves of $\mathcal{F}(\mathbf{V} )$, that pass through the planes $\{P_1 =c<0, Q_1=\theta\}$, exactly once.

The intersection number between each top level curve of $\mathbf{ F}(\mathbf{v}_1)$ and the curve $u_{\mathbb{L}}^{\mathbf{V} }$ is well defined since, as $\mathbf{ v}_1 \neq \mathbf{ 0}$, they are asymptotic to disjoint Lagrangian tori. Moreover all these intersections are positive. We denote the total of these intersection numbers by $\mathbf{ F}(\mathbf{v}_1)\bullet u_{\mathbb{L}}^{\mathbf{V} }$. Similarly, the intersection number of  each top level curve of $\mathbf{ F}(\mathbf{v}_1)$ with any of the planes in either $\mathcal{T}_0$ or $\mathcal{T}_{\infty}$ is well-defined  and all such intersections are positive. Since this number is the same for any plane in the family, we denote these numbers by $\mathbf{ F}(\mathbf{v}_1)\bullet\mathcal{T}_0$ and $\mathbf{ F}(\mathbf{v}_1)\bullet\mathcal{T}_\infty$, respectively.   

Let $\bar{\mathbf{ F}}(\mathbf{v}_1) \colon S^2 \to S^2 \times S^2$ be the compactification of $\mathbf{ F}(\mathbf{v}_1)$,  let $\mathbb{E} \colon (D^2,S^1) \to (S^2 \times S^2, \mathbb{L})$  be the compactification of the curve $u_{\mathbb{L}}^{\mathbf{ V}}$, and let $\bar{\mathcal{T}_0}$ and $\bar{\mathcal{T}}_{\infty}$ be the solid tori obtained by compactifying the planes of $\mathcal{T}_0$ and $\mathcal{T}_{\infty}$. Deforming $\bar{\mathbf{ F}}(\mathbf{v}_1)$ 
arbitrarily close to $\mathbb{L}(\mathbf v_1)$, we obtain a smooth map $F=F(\mathbf{ v}_1) \colon S^2 \to S^2 \times S^2$ such that 
\begin{equation}
\label{sameEs}
F \bullet \mathbb{E} = \bar{\mathbf{ F}}(\mathbf{v}_1) \bullet \mathbb{E} = \mathbf{ F}(\mathbf{v}_1)\bullet u_{\mathbb{L}}^{\mathbf{V} }
\end{equation}
and
\begin{equation}
\label{sameT}
F \bullet \bar{\mathcal{T}_*} = \bar{\mathbf{ F}}(\mathbf{v}_1) \bullet \bar{\mathcal{T}_*} = \mathbf{ F}(\mathbf{v}_1)\bullet {\mathcal{T}_*},\quad \text{ for $*=0,\infty$}.
\end{equation}
Moreover, the corresponding intersection points are identical.

\begin{lemma}\label{FE+-}
Consider
$\mathbf{V} = \{\mathbf{ 0},\mathbf{ v}_1,\mathbf{ v}_2\}=\{(0,0),(a_1,b_1),(a_2,b_2)\}$ such that $\mathbf{ v}_1$ and $\mathbf{ v}_2$ are distinct, $a_1$ is negative, and $b_1$ and $b_2$ are nonzero. Suppose that  $|a_1|$ is sufficiently small with respect to $|b_1|$.\\

\noindent If $b_1>0$, then
$
F \bullet \bar{\mathcal{T}}_0 =0,
$ 
$
F \bullet \bar{\mathcal{T}}_{\infty} =1,
$
and
$
F \bullet \mathbb{E} =\alpha_0.
$
\\

\noindent If $b_1<0$, then 
$
F \bullet \bar{\mathcal{T}}_0 =1,
$
$F \bullet \bar{\mathcal{T}}_{\infty} =0,$
and
$
F \bullet \mathbb{E} =d-1-\alpha_0.
$
\end{lemma}

\begin{proof}
Here we give the proof of the case when $b_1$ is positive. The proof for $b_1<0$ is identical and is left to the reader. The map $F$ represents the class $(1,d)$. For each disk in $\bar{\mathcal{T}}_0$ there is a companion disc in $\bar{\mathcal{T}}_{\infty}$ such that the pair can be glued together, along $\mathbb{L}$, to form a sphere in the class $(0,1)$. Hence, $$F \bullet \bar{\mathcal{T}}_0 + F \bullet \bar{\mathcal{T}}_{\infty} =1.$$ Since all  intersections are positive, in order to prove that  $
F \bullet \bar{\mathcal{T}}_0 =0,
$
and 
$F \bullet \bar{\mathcal{T}}_{\infty} =1,$ it suffices to prove that $F \bullet \bar{\mathcal{T}}_{\infty} \geq 1.$ In particular, in view of \ref{sameT}, it suffices to show that $\underline{u}(\mathbf{ v}_1) \bullet \mathcal{T}_{\infty} \geq1.$

The planes of  $\mathcal{T}_{\infty}$ intersect $\mathcal{U}(\mathbb{L})$ in annuli of the form $\{P_1 = 0, Q_1 = \theta, P_2 > 0\}$.
The curve $\underline{u}(\mathbf{v}_1)$ intersects  $\mathcal{U}(\mathbb{L})$ in the region $\{P_1 >a_1\}$. Since it is essential, $\underline{u}(\mathbf{v}_1)$ must intersect every cylinder of the form $\{P_1=a, Q_1 = \theta \}$ with $a > a_1$.
The curve $\underline{u}(\mathbf{v}_1)$ also has an end asymptotic to a circle in the torus $\mathbb{L}(\mathbf{v}_1)=\{P_1 = a_1, P_2 =b_1\}$.  Since $b_1$ is positive, it follows that for all $a$ sufficiently close to $a_1$, $\underline{u}(\mathbf{v}_1)$ will intersect the annuli  $\{P_1 = a, Q_1 = \theta, P_2 > 0\}$. Hence, if  $|a_1|$ is sufficiently small with respect to $b_1$, then $\underline{u}(\mathbf{v}_1)$ must intersect the planes of $\mathcal{T}_{\infty}$ at least once, as desired.

It remains to prove that $F \bullet \mathbb{E} =\alpha_0$ when $|a_1|$ is sufficiently small with respect to $|b_1|$. By \eqref{sameEs}, and the fact that the top level curves of $\mathbf{ F}(\mathbf{ v}_1)$ are  \begin{equation*}
\label{ }
\{ u_{\mathbb{L}}(\mathbf{v}_1), \underline{u}(\mathbf{v}_1),u_{L_{1,1}},u_1(\mathbf{v}_1), \dots, u_{d-1}(\mathbf{v}_1),\frak{u}_1, \dots, \frak{u}_d\},
\end{equation*}
is suffices to prove that for $|a_1|$ sufficiently small with respect to $|b_1|$, we have
\begin{equation}
\label{ones}
u_{i}(\mathbf{v}_1) \bullet u_{\mathbb{L}}^{\mathbf{V} }=
    1 \quad  \text{ for } 1\leq i \leq \alpha_0, 
\end{equation}
and $u_{\mathbb{L}}^{\mathbf{V} }$ is disjoint from all the other top level curves of $\mathbf{ F}(\mathbf{ v}_1)$.

Recall that $u_{\mathbb{L}}^{\mathbf{V} }$ is an essential curve, and that the image of $p \circ u_{\mathbb{L}}^{\mathbf{V} }$ is $A_0$. So if $w$ is another curve in $S^2 \times S^2$ and $p \circ w$ is disjoint from $A_0$, then $u_{\mathbb{L}}^{\mathbf{V} }$ 
is disjoint from $w$. This observation implies that $u_{\mathbb{L}}^{\mathbf{V} }$ is disjoint from $ u_{L_{1,1}}$ and the $\frak{u}_j$ for $j =1,\dots, d$ since theses curves all project into $A_{\infty}$.

Another consequence of $u_{\mathbb{L}}^{\mathbf{V} }$ being essential with respect to $\mathcal{F}$, is that 
it intersects any fiber of $\mathcal{F}$ either once or not at all. The curve $u_{\mathbb{L}}^{\mathbf{V} }$ intersects  $\mathcal{U}(\mathbb{L})$ in the region $\{P_1<0\}$ and has an end  asymptotic to a circle in $\mathbb{L}=\{P_1 = P_2 =0\}$. Since $b_1 > 0$, this implies that for all $a_1<0$ such that $|a_1|$ is sufficiently small with respect to $b_1$, $u_{\mathbb{L}}^{\mathbf{V} }$ must intersect the annuli of the form $\{P_1 = a_1, Q_1 = \theta, P_2 < b_1\}$ exactly once.
Now the planes $u_{i}(\mathbf{v}_1)$ all belong to broken fibers of $\mathcal{F}$ that intersect $\mathcal{U}(\mathbb{L})$. For $1 \le i \le \alpha_0$, the curves $u_{i}(\mathbf{v}_1)$ intersect $\mathcal{U}(\mathbb{L})$ in annuli of the form $\{P_1 = a_1, Q_1 = \theta, P_2 < b_1\}$. For $i > \alpha_0$, the $u_{i}(\mathbf{v}_1)$ intersect $\mathcal{U}(\mathbb{L})$ in annuli of the form $\{P_1 = a_1, Q_1 = \theta, P_2 > b_1\}$.  Hence,  for $1 \le i \le \alpha_0$, $u_{\mathbb{L}}^{\mathbf{V} }$ intersects the fiber of $\mathcal{F}$ containing $u_{i}(\mathbf{v}_1)$ at a point on $u_{i}(\mathbf{v}_1)$. This yields equation \eqref{ones}. On the other hand, for $i>\alpha_0$, $u_{\mathbb{L}}^{\mathbf{V} }$ intersects the fiber of $\mathcal{F}$ containing $u_{i}(\mathbf{v}_1)$ at a point in the complement of $u_{i}(\mathbf{v}_1)$. Hence, $u_{\mathbb{L}}^{\mathbf{V} }$ is disjoint from these curves. 

Next we show that, when  $|a_1|$ is sufficiently small with respect to $|b_1|$, $u_{\mathbb{L}}^{\mathbf{V} }$ is disjoint from $\underline{u}(\mathbf{ v}_1)$. 
Considering projections, it is clear that the part of $\underline{u}(\mathbf{ v}_1)$ in the complement of $\mathcal{U}(\mathbb{L})$ is disjoint from $u_{\mathbb{L}}^{\mathbf{V} }$ since its projection is contained in the interior of $B \cup A_{\infty}$.

Suppose that $a_1=0$.  Then $\underline{u}((0,b_1)) \cap \mathcal{U}(\mathbb{L})$ is contained in $\{P_1>0\}$ and is asymptotic to $\mathbb{L}(0,b_1)$. This is disjoint from $u_{\mathbb{L}}^{\mathbf{V} }\cap \mathcal{U}(\mathbb{L})$ which is contained in $\{P_1<0\}$ and is asymptotic to $\mathbb{L} =\mathbb{L}(0,0)$. By continuity, $\underline{u}((a_1,b_1)) \cap \mathcal{U}(\mathbb{L})$ is then disjoint from $u_{\mathbb{L}}^{\mathbf{V} }\cap \mathcal{U}(\mathbb{L})$ for all $a_1<0$ with $|a_1|$  sufficiently small with respect to $|b_1|$.

Lastly, we must prove that 
\begin{equation*}
\label{ }
u_{\mathbb{L}}(\mathbf{ v}_1) \bullet u_{\mathbb{L}}^{\mathbf{V} }=0.
\end{equation*}
when $|a_1|$ is sufficiently small with respect to $|b_1|$. Since the compactifications of $u_{\mathbb{L}}^{\mathbf{V} }$ and $u_{\mathbb{L}}$ are homotopic in the space of smooth maps $(D^2, S^1) \to (S^2 \times S^2, \mathbb{L})$, it suffices to show that 
$$ u_{\mathbb{L}}(\mathbf{ v}_1) \bullet u_{\mathbb{L}}=0.$$ 

Let $\bar{u}_{\mathbb{L}}(\mathbf{ v}_1)$ and $\bar{u}_{\mathbb{L}}$ be compactifications of $u_{\mathbb{L}}(\mathbf{ v}_1)$ and $u_{\mathbb{L}}$. We claim that $u_{\mathbb{L}}(\mathbf{ v}_1) \bullet u_{\mathbb{L}}=0$ is equivalent to the fact that the Maslov index of $\bar{u}_{\mathbb{L}}$ is equal to $2$. To see this we recall that 
\begin{equation}
\label{relmas}
\mu(\bar{u}_{\mathbb{L}}) =2 c_1(\bar{u}_{\mathbb{L}})
\end{equation}
where $c_1(\bar{u}_{\mathbb{L}})$ is the relative Chern number of $\bar{u}_{\mathbb{L}}$ which is equal to the number of zeros of a generic section $\xi$ of $\bar{u}_{\mathbb{L}}^*(\Lambda^2 (T(S^2 \times S^2)))$ such that $\xi|_{S^1}$ is nonvanishing and is tangent to $\Lambda^2 (T \mathbb{L})$.

Let $\nu(\bar{u}_{\mathbb{L}})$ be the normal bundle to the embedding $\bar{u}_{\mathbb{L}}$ and fix an identification of $\bar{u}_{\mathbb{L}}^*(T(S^2 \times S^2))$ with the Whitney sum $\nu(\bar{u}_{\mathbb{L}}) \oplus T(D^2)$. For polar coordinates
$(r, \theta)$ on $D^2$ consider the section $r\frac{\partial}{\partial \theta}$ of $\bar{u}_{\mathbb{L}}^*(T(S^2 \times S^2))$. The restriction $r\frac{\partial}{\partial \theta}|_{S^1}$ is nonvanishing and tangent to $T\mathbb{L}$.

Replacing $\mathbf{ v}_1$ by $t\mathbf{ v}_1$ for some small $t>0$, if necessary, we may assume that $\bar{u}_{\mathbb{L}}(\mathbf{ v}_1)$ is close enough $\bar{u}_{\mathbb{L}}$, in the $C^1$-topology, to be identified with a section, $\sigma_{\mathbb{L}}(\mathbf{ v}_1)$,  of  $\nu(\bar{u}_{\mathbb{L}}) \subset \bar{u}_{\mathbb{L}}^*(T(S^2 \times S^2))$. The restriction
$\sigma_{\mathbb{L}}(\mathbf{ v}_1)|_{S^1}$ is roughly parallel to the vector field $\frac{\partial}{\partial P_2}$. By rotating in the normal bundle this section is homotopic through nonvanishing sections to a section of  $T\mathbb{L}$ along $\partial D^2$ which is orthogonal to $\frac{\partial}{\partial \theta}$.  

Set $\xi = r\frac{\partial}{\partial \theta} \wedge \sigma_{\mathbb{L}}(\mathbf{ v}_1)$. It follows from the discussion above that 
$\xi|_{S^1}$ is nonvanishing and is tangent to $\Lambda^2 (T \mathbb{L})$. Moreover, the zeroes of $\xi$ corresponds to the union of the zeros of $r\frac{\partial}{\partial \theta}$ and $\sigma_{\mathbb{L}}(\mathbf{ v}_1)$. Since $\bar{u}_{\mathbb{L}}$ is an embedded the zeros of $\sigma_{\mathbb{L}}(\mathbf{ v}_1)$ exactly correspond to the intersections $u_{\mathbb{L}}(\mathbf{ v}_1) \bullet u_{\mathbb{L}}$. By \eqref{relmas}, we then have 
\begin{equation*}
\label{relmas2}
\mu(\bar{u}_{\mathbb{L}}) =2 (1 + u_{\mathbb{L}}(\mathbf{ v}_1) \bullet u_{\mathbb{L}}).
\end{equation*}
As $\mu(\bar{u}_{\mathbb{L}})=2$ (as it has area $1$ by Lemma \ref{area1}, and $\mathbb{L}$ is monotone) we have $u_{\mathbb{L}}(\mathbf{ v}_1) \bullet u_{\mathbb{L}}=0$, and are done.
\end{proof}

We can deform a building $\mathbf{ G}$ as in Proposition \ref{existence3} within this same framework. As $\mathbf{ G}$ is of Type 3, it's collection of top level curves looks like  
\begin{equation*}
\label{ }
\{v_{\mathbb{L}},\underline{v},v_{L_{1,1}}, v_1, \dots,  v_{d},\frak{v}_1, \dots, \frak{v}_{d-1}\},
\end{equation*}
where 
$v_1, \dots, v_{\gamma_0}$ belong to $\mathcal{T}_0$, $v_{\gamma_0+1}, \dots, v_{d}$ belong to $\mathcal{T}_{\infty}$ and $\frak{v}_j$ belong to $\frak{s}_0 \cup \frak{s}_{\infty}$. Assuming that $\mathbf{ v}_2=(a_2,b_2)$ is sufficiently small we can deform  $\mathbf{ G}$ to obtain a new building $\mathbf{ G}(\mathbf{v}_2)$ with top level curves
 \begin{equation*}
\label{ }
\{ v_{\mathbb{L}}(\mathbf{v}_2), \underline{v}(\mathbf{v}_2),u_{L_{1,1}},v_1(\mathbf{v}_2), \dots, v_{d}(\mathbf{v}_2),\frak{v}_1, \dots, \frak{v}_{d-1}\}.
\end{equation*}
Let $\bar{\mathbf{ G}}(\mathbf{v}_2) \colon S^2 \to S^2 \times S^2$ be the compactification of $\mathbf{ G}(\mathbf{v}_2)$. Again we can deform $\bar{\mathbf{G}}(\mathbf{v}_2)$, 
arbitrarily close to $\mathbb{L}(\mathbf v_2)$, to get a smooth map $G=G(\mathbf{ v}_2) \colon S^2 \to S^2 \times S^2$ such that 
\begin{equation*}
\label{sameE}
G \bullet \mathbb{E} = \bar{\mathbf{ G}}(\mathbf{v}_2) \bullet \mathbb{E} = \mathbf{ G}(\mathbf{v}_2)\bullet u_{\mathbb{L}}^{\mathbf{V} }
\end{equation*}
and
\begin{equation*}
\label{sameT}
G \bullet \bar{\mathcal{T}_*} = \bar{\mathbf{ G}}(\mathbf{v}_2) \bullet \bar{\mathcal{T}_*} = \mathbf{ G}(\mathbf{v}_2)\bullet {\mathcal{T}_*},\quad \text{ for $*=0,\infty$}.
\end{equation*}

 Arguing as in the proof of Lemma \ref{FE+-} we get the following.

\begin{lemma}\label{GE+-}
Consider
$\mathbf{V} = \{\mathbf{ 0},\mathbf{ v}_1,\mathbf{ v}_2\}=\{(0,0),(a_1,b_1),(a_2,b_2)\}$ such that $a_2$ is negative, and $b_1$ and $b_2$ are nonzero. Suppose that  $|a_2|$ is sufficiently small with respect to $|b_2|$.\\

\noindent If $b_2>0$, then
$$
G \bullet \mathbb{E} = \gamma_0 + v_{\mathbb{L}}(\mathbf{ v}_2) \bullet u_{\mathbb{L}}^{\mathbf{V} },
$$
$
G \bullet \bar{\mathcal{T}}_0 =0,
$
and 
$G \bullet \bar{\mathcal{T}}_{\infty} =1.$\\

\noindent If $b_2<0$, then
$$
G \bullet \mathbb{E} = d-\gamma_0 +v_{\mathbb{L}}(\mathbf{ v}_2) \bullet u_{\mathbb{L}}^{\mathbf{V} },
$$
$
G \bullet \bar{\mathcal{T}}_0 =1,
$
and 
$G \bullet \bar{\mathcal{T}}_{\infty} =0.$
\end{lemma}

The term $v_{\mathbb{L}}(\mathbf{ v}_2) \bullet u_{\mathbb{L}}^{\mathbf{V} }$ is not necessarily equal to zero. Instead we have the following identity.

\begin{lemma}\label{noleftg}
For
$\mathbf{V} = \{\mathbf{ 0},\mathbf{ v}_1,\mathbf{ v}_2\}=\{(0,0),(a_1,b_1),(a_2,b_2)\} $ where  $b_1$ and $b_2$ have opposite sign, and $a_1$ and $a_2$ sufficiently small relative to $b_1$ and $b_2$ we have
\begin{equation*}
\label{ }
v_{\mathbb{L}}(\mathbf{ v}_2) \bullet u_{\mathbb{L}}^{\mathbf{V} }=v_{\mathbb{L}}(\mathbf{ v}_2) \bullet u_{\mathbb{L}}(\mathbf{ v}_1).
\end{equation*}

\end{lemma}

\begin{proof}
First we consider the case when $a_1 = a_2 =0$. The image of the map $v_{\mathbb{L}}(\mathbf{ v}_2)$ projects to $A_0$ and its boundary lies in $\mathbb{L}(\mathbf{ v}_2)$. Hence, using our assumption  on sign, the family of Lagrangians $\mathbb{L}(t \mathbf{ v}_1))$ for $0 \le t \le 1$ are disjoint from the compactification of $v_{\mathbb{L}}(\mathbf{ v}_2)$. It then follows from the proof of Lemma \ref{fukaya}, that the compactification of $u_{\mathbb{L}}$ is connected to that of $u_{\mathbb{L}}(\mathbf{ v}_1)$
by a path of smooth maps $u_t \colon (D^2, S^1) \to (S^2 \times S^2, \mathbb{L}(t \mathbf{ v}_1))$. Therefore we have $$v_{\mathbb{L}}(\mathbf{ v}_2) \bullet u_{\mathbb{L}}=v_{\mathbb{L}}(\mathbf{ v}_2) \bullet u_{\mathbb{L}}(\mathbf{ v}_1),$$ as required.

For the general case we use the fact that the maps vary continuously with the parameters and so the intersection numbers remain unchanged for $a_1$ and $a_2$ sufficiently small.


\end{proof}


Since $\mathbf{ v}_1$ and $\mathbf{ v}_2$ are distinct, the intersection numbers of some of the top level curves of $\mathbf{ F}(\mathbf{v}_1)$ and $\mathbf{ G}(\mathbf{v}_2)$ are well-defined. The following results concerning these intersections, will be useful.

\begin{lemma}\label{}
For $\mathbf{ v}_1 =(a_1,b_1)$ and $\mathbf{ v}_2 =(a_2,b_2)$, suppose that $a_1< a_2<0$, $|a_1|$ is sufficiently small with repsect to $|b_1|$, and $|a_2|$ is sufficiently small with repsect to $|b_2|$.\\

\noindent If  $b_1>b_2$, then 
\begin{equation*}
\label{ }
u_i(\mathbf{ v}_1) \bullet v_{\mathbb{L}}(\mathbf{ v}_2) =1 \text{  for } i=1, \dots, \alpha_0
\end{equation*}
and 
\begin{equation*}
\label{ }
v_i(\mathbf{ v}_2) \bullet \underline{u}(\mathbf{ v}_1)=1 \text{ for } i=\gamma_0+1, \dots, d.
\end{equation*}
If  $b_1<b_2$, then 
\begin{equation*}
\label{ }
u_i(\mathbf{ v}_1) \bullet v_{\mathbb{L}}(\mathbf{ v}_2) =1 \text{  for } i=\alpha_0+1, \dots, d-1
\end{equation*}
and 
\begin{equation*}
\label{ }
v_i(\mathbf{ v}_2) \bullet \underline{u}(\mathbf{ v}_1)=1 \text{ for } i=1, \dots, \gamma_0.
\end{equation*}
Moreover, all the intersection points here project to $A_0$.

\end{lemma}

\begin{proof}
Since the curves $\underline{u}(\mathbf{ v}_1)$ and $v_{\mathbb{L}}(\mathbf{ v}_2)$ are essential with respect to $\mathcal{F}$, they intersect a leaf of the foliation either once or not at all. Hence it suffices to detect a single intersection of the relevant pairs of curves listed. We detect an intersection for the first type of pair above and leave the other cases to the reader. For $1 \le i \le \alpha_0$ the planes $u_i(\mathbf{ v}_1)$ intersect $\mathcal{U}(\mathbb{L})$ in annuli $\{P_1 = a_1, Q_1 = \theta, P_2 < b_1\}$. As $v_{\mathbb{L}}((0,b_2))$ is asymptotic to $\mathbb{L}((0,b_2)) = \{P_1 = 0, P_2 = b_2\}$ it intersects $u_i(\mathbf{ v}_1)$ provided $a_1$ is sufficiently small (since the boundary of $v_{\mathbb{L}}((0,b_2))$ intersects all annuli $\{P_1 = 0, Q_1 = \theta, P_2 < b_1\}$). For $a_2$ sufficiently small, the plane $v_{\mathbb{L}}(\mathbf{ v}_2)$ is a deformation of $v_{\mathbb{L}}((0,b_2))$ and so the intersection persists. As $v_{\mathbb{L}}(\mathbf{ v}_2)$ intersects fibers at most once, the intersection number is equal to $1$. Since $a_1<0$, the intersection point projects to $A_0$.
\end{proof}

\begin{corollary}\label{FcapG}
For $\mathbf{ v}_1 =(a_1,b_1)$ and $\mathbf{ v}_2 =(a_2,b_2)$, suppose that $a_1< a_2<0$, $|a_1|$ is sufficiently small with repsect to $|b_1|$, and $|a_2|$ is sufficiently small with repsect to $|b_2|$.\\

\noindent If  $b_1>b_2$, then 
$F \cap G$ contains at least $\alpha_0 + d- \gamma_0$ points in $\mathcal{U}(\mathbb{L})$
that project to $A_0$.\\
If  $b_1<b_2$, then 
$F \cap G$ contains at least $d-1-\alpha_0  +  \gamma_0$ points in $\mathcal{U}(\mathbb{L})$
that project to $A_0$.
\end{corollary}

\begin{remark}
It follows from Lemma \ref{noleftg} that any {\it excess} intersection points between $F$ and $G$ in $\mathcal{U}(\mathbb{L})$ are in bijection with intersection points between $G$ and $\mathbb{E}$, at least if the $b_i$ have opposite sign and the $a_i$ are sufficiently small.
\end{remark}

\subsubsection{Adding deformations near $L_{1,1}$.} To completely resolve the intersections of $F$ and $G$ we must also apply deformations in the Weinstein neighborhood 
$$
\mathcal{U}(L_{1,1}) = \{|p_1|<\epsilon,\, |p_2|<\epsilon\}.
$$
Here we consider nearby Lagrangian tori of the form  $$L_{1,1}(\mathbf{w}) := \{p_1 =c, p_2 =d\},$$ for $\mathbf{w} =(c,d) \in (-\epsilon, \epsilon) \times (-\epsilon, \epsilon).$ The space of 
almost complex structures $\mathcal{J}_{\mathcal{U}(L_{1,1})}$  is defined following the definition of $\mathcal{J}_{\mathcal{U}(\mathbb{L})}$. 

Given  collections 
$$
\mathbf{V} =\{\mathbf{v}_1, \dots, \mathbf{v}_k \} =\{(a_1,b_1), \dots, (a_k,b_k)\},
$$
and 
$$
\mathbf{W} =\{\mathbf{w}_1, \dots, \mathbf{w}_l \} =\{(c_1,d_1), \dots, (c_l,d_l)\}
$$
set $\mathbf{X} =\{\mathbf{V} ,\mathbf{W} \}$. Denote the corresponding almost-complex structure in
$\mathcal{J}_{\mathcal{U}(\mathbb{L})} \cap \mathcal{J}_{\mathcal{U}(L_{1,1})}$
by  $J_{\mathbf{X}}$.  

Lemma \ref{fukaya} generalizes to this setting as follows.

\begin{lemma}\label{fukaya2} Let $u$ be a regular $J$-holomorphic curve with $k\geq0$ ends on $\mathbb{L}$ and $l\geq 0$ ends on $L_{1,1}$.
For all $\mathbf{x}=\{\mathbf{v}, \mathbf{w}\}=\{(a,b), (c,d)\}$  with  $\|\mathbf{x}\|$ sufficiently small, there is a $J_{\mathbf{x}}$-holomorphic curve $u(\mathbf{x})$ that represents the class in $\pi_2(S^2 \times S^2, \mathbb{L}(\mathbf{v}) \cup L_{1,1}(\mathbf{w}) )$ that corresponds to the class  $[u] \in \pi_2(S^2 \times S^2, \mathbb{L}\cup L_{1,1} )$ under the obvious identification. The curve $u(\mathbf{x})$ has $k$ ends on $\mathbb{L}(\mathbf{v})$ and these represent the identical classes in $H_1^{\psi(\mathbf{v})}(\mathbb{L};\ZZ)$ as do those of $u$ in $H_1^{\psi}(\mathbb{L};\ZZ)$. The curve also has  $l$ ends on $L_{1,1}(\mathbf{w})$ which  represent the identical classes in $H_1^{\psi_{1,1}(\mathbf{w})}(L_{1,1};\ZZ)$ as do those of $u$ in $H_1^{\psi_{1,1}}(L_{1,1};\ZZ)$.
\end{lemma}

Corollary \ref{up}  generalizes as follows.

\begin{lemma}
\label{up2}
Let $u_{\mathbb{L}}$ and $u_{L_{1,1}}$ be the essential curves of a building $\mathbf{ F}$ as in Proposition \ref{existence2}. Let $\mathbf{X} =\{\mathbf{V} ,\mathbf{W} \}$ where $$\mathbf{V} =\{(0,0),(a_1,b_1), (a_2,b_2)\}$$ and 
$$\mathbf{W} =\{(0,0),(c_1,d_1), (c_2,d_2)\}.$$
If  $b_1,b_2, d_1, \, d_2$  are in $(-\epsilon, \epsilon)$ and  $a_1,a_2, c_1, c_2$ are in $(-\delta, \delta)$, then for all sufficiently small $\delta$
there is a  $J_{\mathbf{X} }$-holomorphic curve $$u_{\mathbb{L}}^{\mathbf{\mathbf{X} }}\colon \CC \to S^2 \times S^2 \smallsetminus (\mathbb{L}(\mathbf{V} ) \cup L_{1,1}(\mathbf{W} ))$$ in the class of $u_{\mathbb{L}}$ such that $u_{\mathbb{L}}^{\mathbf{\mathbf{X} }}$ is disjoint from the region $\{P_1>0\}$,  the closure of the image of $p\circ u_{\mathbb{L}}^{\mathbf{X} }$ is $ A_0$, and $u_{\mathbb{L}}^{\mathbf{X} }$ intersects the leaves of $\mathcal{F}(\mathbf{X} )$, that pass through the planes $\{P_1 =c<0, Q_1=\theta\}$, exactly once.

There is also a  $J_{\mathbf{X} }$-holomorphic curve $$u_{L_{1,1}}^{\mathbf{X} }\colon \CC \to S^2 \times S^2 \smallsetminus (L(\mathbf{V} ) \cup L_{1,1}(\mathbf{W} ))$$ in the class of $u_{L_{1,1}}$ such that $u_{L_{1,1}}^{\mathbf{X} }$ is disjoint from the region $\{p_1<0\}$,  the closure of the image of $p\circ u_{L_{1,1}}^{\mathbf{X} }$ is $ A_{\infty}$, and $u_{L_{1,1}}^{\mathbf{X} }$ intersects the leaves of $\mathcal{F}(\mathbf{X} )$, that pass through the planes $\{p_1 =c>0, q_1=\theta\}$, exactly once.
\end{lemma}

\medskip

\subsubsection{Completion of the proof of Proposition \ref{intcount}} Let $\mathbf{ F}$ be a building as in  Proposition \ref{existence2} and let $\mathbf{ G}$ be a building as in Proposition \ref{existence3}.  Set
$$\mathbf{x_1} =\{\mathbf{ v}_1,\mathbf{w}_1\}= \{(a_1,b_1),(c_1,d_1)\},$$
$$\mathbf{x_2} =\{\mathbf{ v}_2,\mathbf{w}_2\}= \{(a_2,b_2),(c_2,d_2)\},$$
$$\mathbf{V} =\{\mathbf{ 0},\mathbf{ v}_1,\mathbf{ v}_2\}= \{(0,0),(a_1,b_1),(a_2,b_2)\},$$
$$\mathbf{W} =\{\mathbf{ 0},\mathbf{ w}_1,\mathbf{ w}_2\}= \{(0,0),(c_1,d_1),(c_2,d_2)\},$$
and 
$$ \mathbf{ X} =\{\mathbf{ V},\mathbf{ W}\}.$$
We assume that $\|\mathbf{ x}_1\|$ and $\|\mathbf{x}_2\|$ are small enough for  Lemma \ref{fukaya2} to yield the deformed buildings  $\mathbf{ F}(\mathbf{ x}_1)$ and $\mathbf{ G}(\mathbf{ x}_2)$. 
We also assume that  $|a_1|^2+ |a_2|^2+|c_1|^2 + |c_2|^2$ is small enough with respect to  $|b_1|^2+ |b_2|^2+|d_1|^2 + |d_2|^2$ for  Lemma \ref{up2} to yield the deformations $u_{\mathbb{L}}^{\mathbf{X} }$ and $u_{L_{1,1}}^{\mathbf{X} }$.

Let $\mathbb{E} \colon (D^2,S^1) \to (S^2 \times S^2, \mathbb{L})$ be the 
compactifiction of  $u_{\mathbb{L}}^{\mathbf{X} }$, and $E_{1,1}\colon (D^2,S^1) \to (S^2 \times S^2, L_{1,1})$ be the compactifiction of $u_{L_{1,1}}^{\mathbf{X} }$.  Since the homology classes represented by the ends of $u_{\mathbb{L}}^{\mathbf{X} }$ and $u_{L_{1,1}}^{\mathbf{X} }$ are identical to those of the essential curves $u_{\mathbb{L}}$ and $u_{L_{1,1}}$, the maps $\mathbb{E}$ and $E_{1,1}$ satisfy conditions (2) and (3) of Proposition \ref{intcount}.

Consider compactifications $\bar{\mathbf{ F}}(\mathbf{ x}_1) \colon S^2 \to S^2 \times S^2$ of 
$\mathbf{ F}(\mathbf{ x}_1)$, and $\bar{\mathbf{ G}}(\mathbf{ x}_2) \colon S^2 \to S^2 \times S^2$ of $\mathbf{ G}(\mathbf{ x}_2)$. Arguing as before, we can perturb these maps, arbitrarily close to the Lagrangians $\mathbb{L}(\mathbf{ v}_1)$, $L_{1,1}(\mathbf{ w}_1)$, $\mathbb{L}(\mathbf{ v}_2)$, and $L_{1,1}(\mathbf{ w}_2)$, to obtain smooth spheres $F$ and $G$ such that condition (1) of Proposition \ref{intcount} holds. 

It remains to verify the conditions (4) through (9) of Proposition \ref{intcount} that involve intersections. 

In the current setting, Lemma \ref{FE+-} holds as stated and the proof is unchanged.
\begin{lemma}\label{FE}
Suppose $a_1$ is negative, and $b_1$ and $b_2$ are nonzero. Suppose that  $|a_1|$ is sufficiently small with respect to $|b_1|$.\\

\noindent If $b_1>0$, then 
$
F \bullet \bar{\mathcal{T}}_0 =0,
$ 
$F \bullet \bar{\mathcal{T}}_{\infty} =1,$
and
$
F \bullet \mathbb{E} =\alpha_0.\\
$

\noindent If $b_1<0$, then 
$
F \bullet \bar{\mathcal{T}}_0 =1,
$ 
$
F \bullet \bar{\mathcal{T}}_{\infty} =0,
$
and
$
F \bullet \mathbb{E} =d-1-\alpha_0.
$
\end{lemma}

Lemmas \ref{GE+-} and \ref{noleftg} and Corollary \ref{FcapG} change only in notation and yield the following.

\begin{lemma}\label{GE}
Suppose that $a_2$ is negative, $b_1$ and $b_2$ are nonzero, and $|a_2|$ is sufficiently small with respect to $|b_2|$.\\

\noindent If $b_2>0$, then
$
G \bullet \bar{\mathcal{T}}_0 =0,
$
$G \bullet \bar{\mathcal{T}}_{\infty} =1,$
and
$$
G \bullet \mathbb{E} = \gamma_0 + v_{\mathbb{L}}(\mathbf{x}_2) \bullet u_{\mathbb{L}}^{\mathbf{X} }.\\
$$

\noindent If $b_2<0$, then
$
G \bullet \bar{\mathcal{T}}_0 =1,
$
and 
$G \bullet \bar{\mathcal{T}}_{\infty} =0,$
and 
$$
G \bullet \mathbb{E} = d-\gamma_0 +v_{\mathbb{L}}(\mathbf{x}_2) \bullet u_{\mathbb{L}}^{\mathbf{X} }.
$$
\end{lemma}


\begin{lemma}\label{nog}
If $b_1$ and $b_2$ have opposite sign, and $a_1$ and $a_2$ are sufficiently small, then
\begin{equation*}
\label{ }
v_{\mathbb{L}}(\mathbf{ x}_2) \bullet u_{\mathbb{L}}^{\mathbf{X} }=v_{\mathbb{L}}(\mathbf{ x}_2) \bullet u_{\mathbb{L}}(\mathbf{ x}_1).
\end{equation*}
\end{lemma}

\begin{lemma}\label{FGLL}
Suppose that $a_1< a_2<0$, $|a_1|$ is sufficiently small with respect to $|b_1|$, and $|a_2|$ is sufficiently small with respect to $|b_2|$.\\

\noindent If  $b_1>b_2$, then 
$F \cap G$ contains at least  $\alpha_0 + d- \gamma_0 $ points in $\mathcal{U}(\mathbb{L})$ that project to $A_0$.\\

\noindent If  $b_1<b_2$, then 
$F \cap G \cap \mathcal{U}(\mathbb{L})$ contains at least $d-1-\alpha_0  +  \gamma_0$ points in  $\mathcal{U}(\mathbb{L})$ that project to $A_0$.
\end{lemma}

The following analogous results follow from similar arguments.

\begin{lemma}\label{FE2}
Suppose $c_1$ is positive, $d_1$ and $d_2$ are nonzero, and $|c_1|$ is sufficiently small with respect to $|d_1|$.\\

\noindent If $d_1>0$, then
$
F \bullet \bar{\frak{s}}_0 =0,
$
$
F \bullet \bar{\frak{s}}_{\infty} =1,
$
and 
$
F \bullet E_{1,1} =\beta_0.
$\\

\noindent If $d_1<0$, then
$
F \bullet \bar{\frak{s}}_0 =1,
$
and 
$F \bullet \bar{\frak{s}}_{\infty} =0,$
and
$
F \bullet E_{1,1} =d-\beta_0.
$
\end{lemma}

\begin{lemma}\label{GE2}
Suppose $c_2$ is positive, $d_1$ and $d_2$ are nonzero, and $|c_2|$ is sufficiently small with respect to $|d_2|$.\\

\noindent If $d_2>0$, then
$
G \bullet \bar{\frak{s}}_0 =0,
$
$
G \bullet \bar{\frak{s}}_{\infty} =1
$ 
and
$$
G \bullet E_{1,1} = \delta_0 + v_{L_{1,1}}(\mathbf{x}_2) \bullet u_{L_{1,1}}^{\mathbf{X} }.\\
$$

\noindent If $d_2<0$, then
$
G \bullet \bar{\frak{s}}_0 =1,
$
$
G \bullet \bar{\frak{s}}_{\infty} =0
$ 
and
$$
G \bullet E_{1,1} = d-1-\delta_0 + v_{L_{1,1}}(\mathbf{x}_2) \bullet u_{L_{1,1}}^{\mathbf{X} }.\\
$$

\end{lemma}

\begin{lemma}\label{nog2}
If $d_1$ and $d_2$ have opposite sign, and $c_1$ and $c_2$ are sufficiently small, then
\begin{equation*}
\label{ }
v_{L_{1,1}}(\mathbf{x}_2) \bullet u_{L_{1,1}}^{\mathbf{X} }=v_{L_{1,1}}(\mathbf{ x}_2) \bullet u_{L_{1,1}}(\mathbf{ x}_1).
\end{equation*}
\end{lemma}

\begin{lemma}\label{FGL11}
Suppose that $c_1>c_2>0$, $|c_1|$ is sufficiently small with respect to $|d_1|$, and $|c_2|$ is sufficiently small with respect to $|d_2|$.\\

\noindent If  $d_1>d_2$, then 
$F \cap G $ contains at least $\beta_0 + d -1- \delta_0$ points in $\mathcal{U}(L_{1,1})$ that project to $A_{\infty}$.\\

\noindent If  $d_1<d_2$, then 
$F \cap G $ contains at least $d-\beta_0  +  \delta_0$ points in $\mathcal{U}(L_{1,1})$ that project to $A_{\infty}$.
\end{lemma}

With $F$ and $G$ fixed as above, the remaining analysis can be organized using the following two alternatives.\\

\noindent Alternative 1: either $\alpha_0 \geq \gamma_0$ or $\gamma_0 \geq \alpha_0+1.$\\ \\
\noindent Alternative 2: either $\beta_0 \geq \delta_0+1$ or $\delta_0 \geq \beta_0.$\\

\noindent{\bf Case 1.}  $\alpha_0 \geq \gamma_0$ and $\beta_0 \geq \delta_0+1$. \\

\noindent In this case, we choose our translations so that 
\begin{equation*}
\label{case1}
a_1<a_2<0,\, b_2<0<b_1,\, 0< c_2<c_1,\,\text{and }\, d_2<0 < d_1.
\end{equation*}

For these conditions on $b_1$ and $b_2$, Lemmas \ref{FE} and \ref{GE} yield
$F \bullet \mathcal{T}_0 =0$, $F \bullet \mathcal{T}_{\infty} =1$
$G \bullet \mathcal{T}_0 =1$, and $G \bullet \mathcal{T}_{\infty} =0$. This implies condition (4) of Proposition \ref{intcount}. 

Similarly, for these conditions on $d_1$ and $d_2$, Lemmas \ref{FE2} and \ref{GE2} imply that
$F \bullet \frak{s}_0 =1$, $F \bullet \frak{s}_{\infty} =0$,
$G \bullet \frak{s}_0 =1$, and $G \bullet \frak{s}_0 =1$. This gives condition (5) of Proposition \ref{intcount}.

Since the maps  $F$ and $G$ both represent the class $(1,d)$ in $H_2(S^2 \times S^2;\ZZ))$ we have $F \bullet G = (1,d) \bullet (1,d) = 2d$. On the other hand, for the choices above, Lemmas \ref{FGLL} and \ref{FGL11} imply that 
 \begin{equation*}
\label{in3}
F \bullet G \geq (\alpha_0 + d -\gamma_0) +(\beta_0 + d-1-\delta_0).
\end{equation*}
In the current case, with $\alpha_0 \geq \gamma_0$ and $\beta_0 \geq \delta_0+1$, these two summands are each at least $d$, and so must we have
\begin{equation}
\label{e1}
\alpha_0 = \gamma_0, 
\end{equation}
and
\begin{equation}
\label{e2}
\beta_0 =1+\delta_0.
\end{equation} 
It follows that  $F\cap G$ consists of exactly $2d$ points, $d$ of which are contained in $\mathcal{U}(\mathbb{L})$ and project to $A_0$ and $d$ of which are contained in $\mathcal{U}(L_{1,1})$ and project to $A_{\infty}$. This yields conditions (8) and (9) of Proposition \ref{intcount}.

Since $F \bullet G = \mathbf{ F}(\mathbf{ x}_1) \bullet \mathbf{ G}(\mathbf{ x}_2)$, it follows from the equalities above that there can be no intersections between the essential curves of 
$\mathbf{ F}(\mathbf{ x}_1)$ and those of $\mathbf{ G}(\mathbf{ x}_2)$. In particular, we must have 
\begin{equation}
\label{e3}
v_{\mathbb{L}}(\mathbf{x}_2) \bullet u_{\mathbb{L}}(\mathbf{x}_1)=0,
\end{equation} 
and 
\begin{equation}
\label{e4}
v_{L_{1,1}}(\mathbf{ x}_2) \bullet u_{L_{1,1}}(\mathbf{ x}_1)=0.
\end{equation} 
Equation \eqref{e3} and Lemma \ref{nog} imply that 
\begin{equation*}
\label{ }
v_{\mathbb{L}}(\mathbf{ x}_2) \bullet u_{\mathbb{L}}^{\mathbf{X} }=0.
\end{equation*}
By Lemmas \ref{FE} and \ref{GE} and equation \eqref{e1},  we then have
\begin{equation*}
\label{ }
F \bullet \mathbb{E} + G \bullet \mathbb{E}= \alpha_0 + d -\gamma_0 =d,
\end{equation*}
which yields condition (6) of  Proposition \ref{intcount}.

Similarly, Lemmas \ref{FE2}, \ref{GE2} and \ref{nog2}, together with equations \eqref{e2}  and \eqref{e4}, imply that 
\begin{equation*}
\label{}
F \bullet E_{1,1} + G \bullet E_{1,1} = d
\end{equation*}
and hence condition (7) of Proposition \ref{intcount}. This completes the proof of Proposition \ref{intcount} in the present case.
 
\medskip

The proofs in the other cases follow along identical lines. For the sake of completeness we list the  inequalities for the components of the translations that lead to the desired intersection patterns of Proposition \ref{intcount}, in the remaining scenarios. For the case  $\alpha_0 \geq \gamma_0$ and $\delta_0\geq \beta_0$, we choose  
\begin{equation*}
a_1<a_2<0,\, b_2<0<b_1,\, 0<c_2<c_1,\, \text{ and } d_1<0< d_2.
\end{equation*}
For  $\gamma_0 \geq \alpha_0+1$ and $\beta_0 \geq \delta_0+1$, we choose  
\begin{equation*}
a_1<a_2<0,\, b_1<0<b_2,\, 0<c_2<c_1,\, \text{ and } d_2<0<d_1.
\end{equation*}
Finally for the case $\gamma_0 \geq \alpha_0+1$ and $\delta_0\geq \beta_0$, we choose  
\begin{equation*}
a_1<a_2<0,\, b_1<0<b_2,\, 0<c_2<c_1,\, \text{ and }  d_1<0<d_2.
\end{equation*}

To complete the proof of Proposition \ref{intcount} we remark that the smoothings $F$ and $G$ can be replaced by smooth symplectic spheres without changing the various intersection numbers. To do this, it is enough to replace $F$ and $G$ by symplectic spheres which coincide with $F$ and $G$ away from neighborhoods of $\mathbb{L}(\mathbf v_1)$ and $L_{1,1}(\mathbf w_1)$, respectively $\mathbb{L}(\mathbf v_2)$ and $L_{1,1}(\mathbf w_2)$, that is, the new spheres differ only away from all intersection points.

Now, we know that the asymptotic ends of the top level curves of $\mathbf{ F}(\mathbf{ x}_1)$ and $\mathbf{ G}(\mathbf{ x}_2)$ are simply covered, either because the curves are essential, or for covers of leaves by applying Lemma \ref{area1}. Generically the asymptotic limits are distinct. Then, for small perturbations, we may assume that the top level curves restricted to a neighborhood of the Lagrangians are symplectically isotopic to the corresponding top level curves of our original buildings $\mathbf{F}$ and $\mathbf{G}$. (In the case of $\mathbf{ F}(\mathbf{ x}_1)$ the isotopy maps $\mathbb{L}(\mathbf v_1)$ and $L_{1,1}(\mathbf w_1)$ to $\mathbb{L}$ and $L_{1,1}$ respectively.) Finally
recall that the buildings $\mathbf{F}$ and $\mathbf{G}$ are limits of sequences of smooth embedded holomorphic spheres as our almost complex structures are stretched along the Lagrangians. Therefore, after a small perturbation, we may assume the top level curves of these buildings restricted to a compact subset of the complement of $\mathbb{L} \cup L_{1,1}$ extend to smooth symplectic spheres in $S^2 \times S^2$. Combining the isotopies and these extensions gives our symplectic spheres as required.


\subsection{Scene change}
Consider  $(S^2 \times S^2, \pi_1^*\omega + \pi_2^*\omega)$ equipped with an almost complex structure $J$ adapted to  parameterizations $\psi$ and $\psi_{1,1}$ of $L$ and $L_{1,1}$, respectively. Compactifying the broken leaves of the corresponding foliation $\mathcal{F}$, we get a foliation  $\bar{\mathcal{F}}$ of $S^2 \times S^2$ by spheres. We denote a general sphere in  $\bar{\mathcal{F}}$ by $H$. Let
$F$, $G$, $E$ and $E_{1,1}$ be the spheres and disks  from Proposition \ref{intcount}.
We can work with an almost complex structure with respect to which these spheres and disks are holomorphic. In particular intersections with the leaves of the foliation are all positive.

We start with the following intersection pattern and area profile.
 
\bigskip

\begin{tabular}{ cc }   
Initial intersection numbers & \quad Initial symplectic areas \\  \\

\begin{tabular}{l|l|l|l|l|l|}
\cline{2-6}
                                   & $F$  & $G$   & $\mathbb{E}$ & $E_{1,1}$ & $H$ \\ \hline
\multicolumn{1}{|l|}{$F$}          & 2d    &       &              &           &     \\ \hline
\multicolumn{1}{|l|}{$G$}          & $2d$ & 2d     &              &           &     \\ \hline
\multicolumn{1}{|l|}{$\mathbb{E}$} & $k$  & $d-k$ & $*$            &           &     \\ \hline
\multicolumn{1}{|l|}{$E_{1,1}$}    & $l$  & $d-l$ & 0            & $*$         &     \\ \hline
\multicolumn{1}{|l|}{$H$}          & 1    & 1     & $*$            & $*$         & 0   \\ \hline
\end{tabular} & \quad

\begin{tabular}{l|l|}
\cline{2-2}
                                   & $\pi_1^*\omega + \pi_2^*\omega$--area   \\ \hline
\multicolumn{1}{|l|}{$F$}          & $2+2d$ \\ \hline
\multicolumn{1}{|l|}{$G$}          & $2+2d$ \\ \hline
\multicolumn{1}{|l|}{$\mathbb{E}$} & 1      \\ \hline
\multicolumn{1}{|l|}{$E_{1,1}$}    & 1      \\ \hline
\multicolumn{1}{|l|}{$H$}          & 2      \\ \hline
\end{tabular}\\
\end{tabular}

\bigskip

In this section, we use $F$ and $G$ to alter $(S^2 \times S^2, \pi_1^*\omega + \pi_2^*\omega)$, away from $\mathbb{L}$ and $L_{1,1}$, to derive a scenario in which the disjointness of these Lagrangians is a contradiction. The spheres $F$ and $G$ intersect one another in $2d$ points,  $\{p_1, \dots, p_{2d}\}$. 
By positivity of intersection, the homology classes involved imply that each of the $p_i$ in $F \cap G$ lies on a sphere, $H_i$, of $\bar{\mathcal{F}}$, and the $H_i$ are distinct. 
We may also assume that for some fixed $\epsilon>0$ there are disjoint Darboux balls $B_i$ of capacity $\epsilon$ around each $p_i$ on which $J$  is standard and the $F$, $G$ and $H_i$ restrict to planes through the origin. 

\medskip

\noindent {\em Step 1:} Blow up the balls $B_i$ of capacity $\epsilon$ around each of the $p_i$.

\medskip

Denote the new manifold by $(W, \Omega_1)$. It follows from the analysis of the blowup procedure from \cite{mcp}, see also Proposition 9.3.3 of \cite{mcs}, that  $(W, \Omega_1)$ contains $2d$ exceptional divisors $\mathcal{E}_i$ each of area $\epsilon$. Since the $H_i$ are $J$-holomorphic in each $B_i$, $(W, \Omega_1)$ also contains the proper transforms of the $H_i$. These are denoted here by $\hat{H}_i$ and are symplectic spheres  of area $2-\epsilon$. By property (9) of Proposition \ref{intcount}, $d$ of the $\hat{H}_i$ intersect $\mathbb{E}$ once, and the other $d$ of the $\hat{H}_i$ intersect $E_{1,1}$ once. (Note that we may assume the families of planes $\mathcal{T}_0$ and $\mathcal{T}_{\infty}$, and also the families 
$\frak s_0$ and $\frak s_{\infty}$, are still $J$-holomorphic after our perturbation of $J$ since the smoothing of $F$ and $G$ occurs away from our broken planes. Therefore $p^{-1}(A_0)$ and  $p^{-1}(A_{\infty})$ remain the same sets as in Proposition \ref{intcount} (9).)

The proper transforms of $F$ and $G$, denoted by $\hat{F}$ and $\hat{G}$, are also well-defined. These are spheres of area $2d+2-2d\epsilon$ which are now disjoint. Every sphere $H$ of $\bar{\mathcal{F}}$, other than the $H_i$, has a proper transform $\hat{H}$ of area two. 

This information is collected in the following tables.
%
%
%

\medskip
\begin{tabular}{ cc }   
Intersection numbers after Step 1& \quad Areas after Step 1\\  \\
\begin{tabular}{l|l|l|l|l|l|l|l|}
\cline{2-8}
                                      & $\hat{F}$ & $\hat{G}$ & $\mathbb{E}$ & $E_{1,1}$ & $\hat{H}$ & $\mathcal{E}_i$ & $\hat{H}_i$ \\ \hline
\multicolumn{1}{|l|}{$\hat{F}$}       & 0         &           &              &           &                 &             &   \\ \hline
\multicolumn{1}{|l|}{$\hat{G}$}       & 0         & 0         &              &           &                 &             &   \\ \hline
\multicolumn{1}{|l|}{$\mathbb{E}$}    & $k$       & $d-k$     & $*$           &           &                 &             &   \\ \hline
\multicolumn{1}{|l|}{$E_{1,1}$}       & $l$       & $d-l$     & 0            & $*$         &                 &             &   \\ \hline
\multicolumn{1}{|l|}{$\hat{H}$}       & $1$         & $1$         & $*$            & $*$         & 0               &             &   \\ \hline
\multicolumn{1}{|l|}{$\{\mathcal{E}_i\}$} & $2d$         & $2d$         &    0          &     0      & 0              & -1           &   \\ \hline
\multicolumn{1}{|l|}{$\{\hat{H}_i\}$}     &  0         &  0         &   $ d$          &       $d$     &    0             &     1        & -1 \\ \hline
\end{tabular} & \quad
\begin{tabular}{l|l|}
\cline{2-2}
                                      & $\Omega_1$--Area                \\ \hline
\multicolumn{1}{|l|}{$\hat{F}$}       & $2+2d -2d\epsilon$      \\ \hline
\multicolumn{1}{|l|}{$\hat{G}$}       & $2+2d-2d\epsilon$       \\ \hline
\multicolumn{1}{|l|}{$\mathbb{E}$}    & $1$                     \\ \hline
\multicolumn{1}{|l|}{$E_{1,1}$}       & $1$                     \\ \hline
\multicolumn{1}{|l|}{$\hat{H}$}       & $2$                     \\ \hline
\multicolumn{1}{|l|}{$\mathcal{E}_i$} & $\epsilon$ \\ \hline
\multicolumn{1}{|l|}{$\hat{H}_i$}     & $2-\epsilon$            \\ \hline
\end{tabular}\\
\end{tabular}

\bigskip

\noindent {\em Step 2:} Inflate both $\hat{F}$ and $\hat{G}$ by adding a tubular neighborhood of capacity $d$.

\medskip

For the inflation result from \cite{lm96}, we may assume by Lemma 3.1 in \cite{mcd01} that the new form tames our original almost complex structure. In particular all of our holomorphic curves remain holomorphic and in particular symplectic.

Denoting the resulting symplectic form on $W$ by $\Omega_2$ we get the following new symplectic area profile.

\bigskip

\begin{center}
\begin{tabular}{l|l|}
\cline{2-2}
                                      & $\Omega_2$--Area                    \\ \hline
\multicolumn{1}{|l|}{$\hat{F}$}       & $2+2d -2d\epsilon$      \\ \hline
\multicolumn{1}{|l|}{$\hat{G}$}       & $2+2d-2d\epsilon$       \\ \hline
\multicolumn{1}{|l|}{$\mathbb{E}$}    & $1+d^2$                     \\ \hline
\multicolumn{1}{|l|}{$E_{1,1}$}       & $1+d^2$                     \\ \hline
\multicolumn{1}{|l|}{$\hat{H}$}       & $2+2d$                     \\ \hline
\multicolumn{1}{|l|}{$\mathcal{E}_i$} & $\epsilon+2d$ \\ \hline
\multicolumn{1}{|l|}{$\hat{H}_i$}     & $2-\epsilon$            \\ \hline
\end{tabular}
\end{center}

\bigskip

\noindent {\em Step 3:} Apply the negative inflation procedure from \cite{bu}, of size $\epsilon$, to each $\mathcal{E}_i$.

\medskip

Denoting the resulting symplectic form on $W$ by $\Omega_3$, the area profile becomes

\bigskip
\begin{center}
\begin{tabular}{l|l|}
\cline{2-2}
                                      & $\Omega_3$--Area                    \\ \hline
\multicolumn{1}{|l|}{$\hat{F}$}       & $2+2d $      \\ \hline
\multicolumn{1}{|l|}{$\hat{G}$}       & $2+2d$       \\ \hline
\multicolumn{1}{|l|}{$\mathbb{E}$}    & $1+d^2$                     \\ \hline
\multicolumn{1}{|l|}{$E_{1,1}$}       & $1+d^2$                     \\ \hline
\multicolumn{1}{|l|}{$\hat{H}$}       & $2+2d$                     \\ \hline
\multicolumn{1}{|l|}{$\mathcal{E}_i$} & $2d$ \\ \hline
\multicolumn{1}{|l|}{$\hat{H}_i$}     & $2$            \\ \hline
\end{tabular}
\end{center}

\bigskip

\noindent {\em Step 4:} Blow down each $\hat{H}_i$.
\medskip

We denote the symplectic manifold resulting from this final step by $(X, \Omega)$. Each of the exceptional divisors $\mathcal{E}_i$ in $(W,\Omega_3)$ is transformed,  by Step 4, into a sphere $\mathcal{H}_i$ in $X$ which has $\Omega$-area equal to $2d +2$ and now lies in the same class as the $\hat{H}$. The disks $\mathbb{E}$ and $E_{1,1}$ each intersect $d$ of the $\hat{H}_i$ and so are transformed by Step $4$ into disks $\mathbb{E}^X$ and $E_{1,1}^X$, whose symplectic areas have each been increased by $2d$.

\bigskip

\begin{tabular}{ cc }   
Intersection numbers after Step 4& \quad Areas after Step 4\\  \\
\begin{tabular}{l|l|l|l|l|l|l|l}
\cline{2-7}
                                      & $\hat{F}$ & $\hat{G}$ & $\mathbb{E}^X$ & $E_{1,1}^X$ & $\hat{H}$ & $\mathcal{H}_i$ \\ \hline
\multicolumn{1}{|l|}{$\hat{F}$}       & 0         &           &              &           &                 &       \\ \hline
\multicolumn{1}{|l|}{$\hat{G}$}       & 0         & 0         &              &           &                 &         \\ \hline
\multicolumn{1}{|l|}{$\mathbb{E}^X$}    & $k$       & $d-k$     & $*$           &           &           &          \\ \hline
\multicolumn{1}{|l|}{$E_{1,1}^X$}       & $l$       & $d-l$     & 0            & $*$         &                 &            \\ \hline
\multicolumn{1}{|l|}{$\hat{H}$}       & $1$         & $1$         & $*$            & $*$         & 0               &             \\ \hline
\multicolumn{1}{|l|}{$\{\mathcal{H}_i\}$} & $2d$         & $2d$         &    $d$          &     $d$      & 0              &  0  \\ \hline
\end{tabular} & \quad
\begin{tabular}{l|l|}
\cline{2-2}
                                      & $\Omega$--Area                \\ \hline
\multicolumn{1}{|l|}{$\hat{F}$}       & $2+2d$      \\ \hline
\multicolumn{1}{|l|}{$\hat{G}$}       & $2+2d$       \\ \hline
\multicolumn{1}{|l|}{$\mathbb{E}^X$}    & $1+d^2+2d$                     \\ \hline
\multicolumn{1}{|l|}{$E_{1,1}^X$}       & $1+d^2+2d$                     \\ \hline
\multicolumn{1}{|l|}{$\hat{H}$}       & $2+2d$                     \\ \hline
\multicolumn{1}{|l|}{$\mathcal{H}_i$} & $2+2d$ \\ \hline
\end{tabular}\\
\end{tabular}

\bigskip

\begin{lemma}\label{s2} $(X, \Omega)$ is symplectomorphic to $$(S^2 \times S^2, (d+1)\omega \oplus (d+1)\omega).$$
\end{lemma}

\begin{proof}
The presence of the embedded symplectic spheres $\hat{F}$ and $\hat{H}$, which have the same $\Omega$-area and satisfy
$$  
\hat{F} \bullet \hat{F} = \hat{H} \bullet \hat{H} =0, \text{  and  } \hat{F} \bullet \hat{H} =1,
$$
implies that either $(X, \omega)$ is symplectomorphic to $$(S^2 \times S^2, (d+1)\omega \oplus (d+1)\omega)$$ or there are finitely many symplectically embedded spheres with self-intersersection number $-1$ in the complement of $\hat{F}$ and $\hat{H}$ in $X$, and $X$ can be blown down to a copy of $S^2 \times S^2$. This follows from the proof of Theorem 9.4.7 of \cite{mcs}. As a consequence if $H_2(X;\ZZ)$ has rank $2$ then $X$ is symplectomorphic to $S^2 \times S^2$.

A simple analysis of the construction of $(X,\Omega)$ from $(S^2 \times S^2, \pi_1^*\omega + \pi_2^*\omega)$ allows us to compute this rank. The $2d$ blow ups in Step 1 imply that the rank of $H_2(W;\ZZ)$ is $2+2d$. The subsequent $2d$ blow down operations  in Step 4 imply that the rank of $H_2(X;\ZZ)$ is $2$ as required.
\end{proof}

Henceforth, we may  identify $(X, \Omega)$ with $(S^2 \times S^2, (d+1)\omega \oplus (d+1)\omega)$.
The Lagrangian tori $\mathbb{L}$ and $L_{1,1}$ are untouched, as submanifolds, by the four steps above.  They remain Lagrangian and disjoint in $(X,\Omega)$. Note that $L_{1,1}$ is not equal to the Clifford torus in $(X, \Omega)$ with respect to the identification above. In what follows we denote the Clifford torus in $(X, \Omega)$ by $L_X$.

The manifold $(X,\Omega)$ also inherits an almost complex structure,  denoted here by $\hat{J}$,  which equals $J$ away from the collection $\{\hat{\mathcal{H}}_i\}$. In particular, $\hat{J}$ is adapted to the original parameterizations $\psi$ and $\psi_{1,1}$ of $L$ and $L_{1,1}$. As in Section \ref{double}, $\hat{J}$ determines a straightened foliation $\hat{\mathcal{F}}= \mathcal{F}(\mathbb{L},L_{1,1}, \psi,\psi_{1,1}, \hat{J})$ of $X \smallsetminus (\mathbb{L} \cup L_{1,1})$. The original collections of planes 
$\frak{s}_0$, $\frak{s}_{\infty}$, $\mathcal{T}_0$ and $\mathcal{T}_{\infty}$  still comprise the broken leaves of this new foliation. The symplectic spheres $\hat{F}$ and $\hat{G}$ now represent the class  $(1,0) \in H_2(X;\ZZ) = H_2(S^2 \times S^2; \ZZ)$. As in Proposition \ref{intcount}, it is still true that exactly one of $\hat{F}$ and $\hat{G}$ intersects  the planes of $\frak{s}_{0}$ and the other intersects the planes of $\frak{s}_{\infty}$, and 
 exactly one of $\hat{F}$ and $\hat{G}$ intersects  the planes of  $\mathcal{T}_{0}$ and the other intersects  the planes of $\mathcal{T}_{\infty}$. 
%
%

\begin{lemma}
The Lagrangian tori $\mathbb{L}$ and $L_{1,1}$ are both monotone in $(X, \Omega)$.
\end{lemma}

\begin{proof}
Let $D_{\infty} \colon (D^2, S^1) \to (S^2 \times S^2, \mathbb{L})$ be a compactification of one of the planes of $\mathcal{T}_\infty$. The disk $D_{\infty}$ has Maslov index equal to $2$ and 
symplectic area equal to $1$ with respect to $\pi_1^*\omega + \pi_2^*\omega$. The map $D_{\infty}|_{S^1}$ represents the foliation class $\beta_{\mathbb{L}}$. The image of the map $D_{\infty}$ is unaffected by the four steps defining the passage from  $(S^2 \times S^2, \pi_1^*\omega + \pi_2^*\omega)$ to $(X, \Omega)$. Viewed  as a map from $(D^2, S^1)$ to $(X, \mathbb{L})$, $D_{\infty}$ still has Maslov index 2, and $D_{\infty}|_{S^1}$ still represents $\beta_{\mathbb{L}}$. The $\Omega$--area of $D_{\infty}$, as a map into $(X, \Omega)$, is $d+1$. This follows from the fact that exactly one of $F$ and $G$ intersect $D_{\infty}$ and so the inflations in Step 2 increase the symplectic area by $d$.

By assertion (4) of Proposition \ref{intcount}, the boundary $\mathbb{E}|_{S^1}$ represents a class which together with $\beta_{\mathbb{L}}$ forms an integral basis of $H_1(\mathbb{L};\ZZ).$ The same holds for 
$\mathbb{E}^X|_{S^1}$. To prove that $\mathbb{L}$ is a monotone Lagrangian torus in $(X, \Omega)$ it then suffices, by Lemma \ref{pair}, to prove that the Maslov index of $\mathbb{E}^X \colon (D^2, S^1) \to (X, \mathbb{L})$ is equal to $$\frac{2}{d+1}(1+d^2 +2d) = 2d+2$$ where $1+d^2+2d$ is the area of  $\mathbb{E}^X$.  This follows from the fact that, in $(W, \Omega_3)$, $\mathbb{E}$ has Maslov index $2$,  intersects exactly $d$ of the $\hat{H}_i$, and each of the corresponding intersection numbers is $1$. In blowing down the $\hat{H}_i$, and passing from $\mathbb{E}$ to $\mathbb{E}^X$, each of these intersection points yields an increase of $2$ in the Maslov index.

The proof  that $L_{1,1}$ is monotone in $(Y, \Omega)$ is identical.
\end{proof}

%

\begin{lemma}
The Lagrangians $\mathbb{L}$ and $L_{1,1}$ are both Hamiltonian isotopic to the Clifford torus $L_X$ in $(X, \Omega)$. 
\end{lemma}

\begin{proof}
This follows from the main result of Cieliebak and  Schwingenheuer in \cite{cisc}. In the language of that paper the compactification of the straightened foliation $\hat{\mathcal{F}}= \mathcal{F}(\mathbb{L},L_{1,1}, \psi,\psi_{1,1}, \hat{J})$ yields a fibering of $\mathbb{L}$ and  a fibering of $L_{1,1}$.
For the fibering of $\mathbb{L}$, the spheres $\hat{F}$ and $\hat{G}$ are disjoint sections in the class $(1,0)$ and exactly one of them  intersects  the (compactification of the ) planes of  $\mathcal{T}_{0}$ and the other intersects  the those of $\mathcal{T}_{\infty}$. The main theorem of \cite{cisc}, then implies that $L$ is Hamiltonian isotopic to the Clifford torus $L_X$ in $(X, \Omega)$. An identical argument holds for $L_{1,1}.$
\end{proof}

With this, the contradiction to Assumption 2 becomes apparent. The first fundamental intersection result implied by the Floer theory for monotone Lagrangian submanifolds implies that any Lagrangian tori Hamiltonian isotopic to $L_X$ must intersect nontrivially, \cite{oh}. Hence, $\mathbb{L}$ and $L_{1,1}$ can not be disjoint in $(X, \Omega)$.

\begin{remark}
The assumption that $\mathbb{L}$ and $L_{1,1}$ are disjoint is used twice in the proof of Theorem \ref{one}. At the very end, and  in the proof of Refinement 3 in Section \ref{claim3}
\end{remark}

\begin{remark}
The fact that $L_{1,1}$ is the Clifford torus (and not just another monotone Lagrangian torus) is crucial (only) in the proof of the existence results Proposition \ref{existence2} and Proposition \ref{existence3}. 
\end{remark}

\begin{remark}
There is an alternative to the argument used at the end of the proof of Theorem \ref{one} that avoids appealing to Lagrangian Floer homology. Instead one can use the fact that the symplectomorphism in Lemma \ref{s2} can be chosen to map $\hat{F}$, $\hat{G}$ and the transforms $\hat{T}_0$ and $\hat{T}_{\infty}$ to the axes $S^2 \times \{0\}$, $S^2 \times \{\infty\}$, $\{0\} \times S^2$ and $\{ \infty\} \times S^2$ respectively. The complement of these axes in $S^2 \times S^2$ can be identified with a domain in $T^* T^2$, in which the Clifford torus is identified with the zero section. We can check  that $\mathbb{L}$ and $L_{1,1}$ are homologically nontrivial in this copy of $T^* T^2$ and so, by Theorem \ref{hom}, are Hamiltonian isotopic to constant sections. The monotonicity condition then implies the constant section must be the zero section. Finally Gromov's intersection theorem for exact Lagrangians in cotangent bundles, from Section $2.3.B''_4$ of \cite{gr}, implies that they must intersect. 
\end{remark}

\section{Proof of Theorem \ref{two}.}

It suffices to prove the following.

\begin{theorem}\label{eps}
For any $\epsilon>0$ there is a $\delta >0$ and a symplectic embedding of the  polydisk $P(1+ \delta,1 + \delta)$ into $P(2+ \epsilon, 2+ \epsilon)$ whose image is disjoint from the product Lagrangians $L_{k,l}$ for $k,l \in \{1,2\}$. 
\end{theorem}
The desired additional integral Lagrangian torus $L^+$ is the one on (the image of) the boundary of $P(1,1) \subset P(1+ \delta, 1+ \delta)$.

\subsection{Proof of Theorem \ref{eps}}
We will use rescaled polar coordinates $\theta_i, R_i$ on $\RR^4=\CC^2$ where $R_i = \pi |z_i|^2$ and $\theta_i \in \RR \slash \ZZ$.  In these coordinates the standard symplectic form is $$\omega = \sum_{i=1}^2 dR_i \wedge d \theta_i$$ and $L_{k,l} = \{(\theta_1,k,\theta_2,l)\}$.

\subsubsection{A polydisk.} For $\epsilon>0$ fixed, choose positive numbers $\ell$, $w$ such that \begin{itemize}
  \item $2< \ell < 2+\epsilon$
  \item $ w < 2$
   \item $\frac{1}{\ell} + \frac{1}{w} <1$.
\end{itemize}
Then choose positive constants $\sigma$ and $\delta$ such that 
\begin{itemize}
  \item $\ell + \sigma < 2+\epsilon$
  \item $w +\sigma <2$
  \item $\frac{1+ \delta}{\ell} + \frac{1+ \delta}{w} <1$
\end{itemize}
Set $$S = \{ \sigma < R_1 < \ell + \sigma, \sigma < R_2 < w + \sigma\}$$ 
 and$$T=\left\{ 0 < \theta_1 < \frac{1+ \delta}{\ell}, 0 < \theta_2 < \frac{1+ \delta}{w}\right\}.$$ 
Note that $S \times T$ is a subset of  $P(2+ \epsilon, 2+ \epsilon)$ and is symplectomorphic to $P(1+\delta, 1+ \delta)$. Both $L_{1,1}$ and $L_{2,1}$ intersect $S \times T$, while $L_{1,2}$ and $L_{2,2}$ do not.

\subsubsection{The plan} To prove Theorem \ref{eps} it suffices to find a Hamiltonian diffeomorphism of $P(2+ \epsilon, 2+ \epsilon)$ that displaces 
$S\times T$ from the $L_{k,l}$. Equivalently, we construct a Hamiltonian diffeomorphism $\Psi$ of $P(2+ \epsilon, 2+ \epsilon)$ such that each of the images  $\Psi(L_{k,l})$ is disjoint from $S \times T$.

To construct $\Psi$ we use Hamiltonian functions which are of the form $F(\theta_1, \theta_2)$. The Hamiltonian flow, $\phi^t_F$, of such a function preserves $\theta_1$ and $\theta_2$ and generates a Hamiltonian vector field parallel to the $R_1R_2$--plane. In particular,  the only points of $\phi^t_F(L_{k,l})$ which could possibly intersect $S \times T$ are those whose $(\theta_1, \theta_2)$ coordinates lie in $T$. 

Since we only need to control the images of the $L_{k,l}$,  we can cut off autonomous functions like $F$ in (moving) neighborhoods of $\phi^t_F(L_{k,l})$ for specific values of $k$ and $l$. After this cutting off,  the new Hamiltonian will depend on all variables and be time dependent. In general, for a closed subset $V$, we denote  the function obtained by cutting of $F$ along $\phi^t_F(V)$ by $F_{[V]}$. Note that $$\phi^t_{F_{[V]}}(v) =\phi^t_F(v),\,\,\text{for all }\,\, v \in V \,\,\text{and}\,\, t \in [0,1].$$
As well, each map $\phi^t_{F_{[V]}}$ is equal to the identity away from an arbitrarily small neighborhood of $$\bigcup_{t \in [0,1]}\phi^t_F(V).$$

\subsubsection{A diagonal move}
Let $g \colon \RR/\ZZ \to \RR$ be a smooth function such that for some positive real number $c(g)>0$ we have $$g'(s) =c(g),\, \text{for} \,\,s \in \left[0, \frac{1+ \delta}{\ell} + \frac{1+ \delta}{w}\right],$$ $\max (g') =c(g)$, and $\min (g')$ is less than and arbitrarily close to
$$-c(g)\left(\frac{\frac{1+ \delta}{\ell} + \frac{1+ \delta}{w}}{1- \frac{1+ \delta}{\ell} -\frac{1+ \delta}{w}}\right).$$
Letting  $G(\theta_1, \theta_2) =g(\theta_1+\theta_2)$, we have 
\begin{equation}
\label{gflow}
\phi^t_G(\theta_1,R_1, \theta_2, R_2) = (\theta_1,R_1+t g'(\theta_1+\theta_2), \theta_2 , R_2+t g'(\theta_1+\theta_2))
\end{equation}
The image $\phi^1_G(L_{1,2})$ is well defined as long as 
\begin{equation}
\label{constraint}
c(g) < \frac{1- \frac{1+ \delta}{\ell} -\frac{1+ \delta}{w}}{\frac{1+ \delta}{\ell} + \frac{1+ \delta}{w}}
\end{equation}
and is contained in $P(2+ \epsilon, 2+ \epsilon)$ as long as $c(g) < \epsilon$. 
Henceforth, we will assume that $\ell$, $w$ and $\delta$ have been chosen  such that the first constraint on $c(g)$, implies the second.

It follows from \eqref{gflow} and \eqref{constraint} that $\phi^t_G(L_{1,2})$ is contained in $$\{R_1 \leq 1+ c(g)\} \cap \{R_2 > 1\}$$ for all $t \in [0,1]$. Hence, each image $\phi^t_G(L_{1,2})$ is disjoint from the other $L_{k,l}$. Since $g' =c(g)>0$ on $T$, 
each $\phi^t_G(L_{1,2})$ is also disjoint from $S \times T$.

\subsubsection{A vertical move.}

Let $h \colon \RR/\ZZ \to \RR$ be a smooth function such that for some positive real number $0< c(h) < \sigma$ we have $$h'(s) =-(1-c(h)),\, \text{for} \,\,s \in \left[0, \frac{1+ \delta}{w}\right],$$ $\min(h') =-(1-c(h))$,  and $\max(h')$ is greater than and arbitrarily close to
$$\frac{(1-c(h))\left(\frac{1+ \delta}{w}\right)}{1 -\frac{1+ \delta}{w}} =\frac{1-c(h)}{\frac{w}{1+ \delta}-1}$$
which is greater than one since $w+\sigma<2$ and $c(h)< \sigma$. 

Letting $H(\theta_1, \theta_2) =h(\theta_2)$, we have
\begin{equation}
\label{hflow}
\phi^t_H(\theta_1,R_1, \theta_2, R_2) = (\theta_1,R_1, \theta_2 , R_2+th'(\theta_2)).
\end{equation}
Clearly, $L_{2,1}$ and $L_{2,2}$ are disjoint from $\phi^t_H(L_{1,1})$ for all $t \in [0,1]$.
Moreover, for $\theta_2$ in $\left[0, \frac{1+ \delta}{w}\right]$ we have 
\begin{equation*}
\label{ }
\phi^1_H(\theta_1,1, \theta_2, 1) = (\theta_1,1, \theta_2 ,c(h)).
\end{equation*}
So  $\phi^1_H(L_{1,1})$ is disjoint from $T \times S$ by our choice of $c(h)$.

Some points of $L_{1,1}$, with values of $\theta_2$ in $\left(\frac{1+ \delta}{w},1\right)$, are mapped by $\phi^1_H$ to points having $R_2$ coordinate greater than and arbitrarily close to 
$$1 + \frac{1-c(h)}{\frac{w}{1+ \delta}-1}>2.$$
Choosing $w$ sufficiently close to $2$, and $\delta$ sufficiently small we can ensure that $\phi^1_H(L_{1,1})$ lies in $P(2+ \epsilon, 2+ \epsilon)$. 

\subsubsection{A time delay}

The Hamiltonian diffeomorphism $\phi^{1}_{H_{[L_{1,1}]}}$ can not be used to move  $L_{1,1}$  off of $S \times T$ while leaving  $L_{1,2}$ undisturbed. For, as described in the discussion above, $\phi^{1}_{H_{[L_{1,1}]}}(L_{1,2})$ will intersect $S \times T$.

The Hamiltonian diffeomorphism $$\phi^{1}_{H_{[L_{1,1}]}} \circ \phi^1_{G_{[L_{1,2}]}}$$ has the same problem.  By \eqref{gflow} and \eqref{hflow}, the image of $(\theta_1,1, \theta_2, 1) \in L_{1,1}$ under $\phi^t_H$ belongs to $\phi^1_G(L_{1,2})$ if and only if $g'(\theta_1+\theta_2)=0$ and $th'(\theta_2)=1$. Since $\max(h')>1$, these intersections occur and so the map above will again push $L_{1,2}$ into $S \times T.$

We can fix this by adding a time delay.
The first intersection  between $\phi^t_H(L_{1,1})$ and $\phi^1_G(L_{1,2})$ occurs at $t = (\max(h') )^{-1}$. Let $\tau$ be  less than and arbitrarily close to $(\max(h') )^{-1}$. Hence, $\tau$ is also  less than and arbitrarily close to 
\begin{equation*}
\label{ }
\frac{\frac{w}{1+ \delta}-1}{1-c(h)}.
\end{equation*}
Consider the Hamiltonian diffeomorphism  $$\tilde{\Psi}= \phi^{1-\tau}_{H_{\left[\phi^{\tau}_{H}(L_{1,1}) \cup \phi^1_{G}(L_{1,2})\right]} }\circ \phi^{\tau}_{H_{[L_{1,1}]}} \circ \phi^1_{G_{[L_{1,2}]}}.$$
It follows from the analysis above, that the map $\tilde{\Psi}$ is compactly supported in $P(2+ \epsilon, 2+ \epsilon)$. In fact, it is supported in an arbitrarily small neighborhood of the subset $\{R_1 \leq 1+ c(g)\}.$ Hence,  $\tilde{\Psi}(L_{2,1}) = L_{2,1}$ and  $\tilde{\Psi}(L_{2,2}) = L_{2,2}$.  By the definitions of $\tau$ and the cut-off operation we have  $\tilde{\Psi}(L_{1,1}) = \phi^1_H(L_{1,1})$ and thus $\tilde{\Psi}(L_{1,1})$ is disjoint from $S \times T$. In addition, we now have the following.

\begin{lemma}\label{off} The image $\tilde{\Psi}(L_{1,2})$ is disjoint from $S \times T$ when $c(h)$ is sufficiently close to $\sigma$ and $\delta$ is sufficiently small.
\end{lemma}

\begin{proof}
By construction, for $(\theta_1,\theta_2) \in T$ we have 
\begin{eqnarray*}
\tilde{\Psi}(\theta_1,1, \theta_2, 2) & = &(\theta_1,1+ g'(\theta_1+\theta_2), \theta_2 , 2+ g'(\theta_1+\theta_2)+ (1-\tau)h'(\theta_2)) \\
{} & = & (\theta_1,1+ c(g), \theta_2 , 2+ c(g)-(1-\tau)(1-c(h))). 
\end{eqnarray*}
It suffices to show that we can choose $c(g)$ and $c(h)$ so that 
\begin{equation}
\label{need}
2+ c(g)-(1-\tau)(1-c(h)) > w + \sigma.
\end{equation}
Since $\tau$ is less than and arbitrarily close to
\begin{equation*}
\label{ }
\frac{\frac{w}{1+ \delta}-1}{1-c(h)},
\end{equation*}
is also suffices to show that we can choose $c(g)$ and $c(h)$ so that 
\begin{equation*}
\label{ }
c(g) > w\left(1- \frac{1}{1+\delta}\right) +(\sigma-c(h)).
\end{equation*}
The righthand side can be made arbitrarily small by taking $c(h)$ to be close to $\sigma$ and $\delta$ to be small. Since the choice of $c(g)$ is independent of the choice of $c(h)$ and the constraint \eqref{constraint} on $c(g)$ relaxes as $\delta$ goes to zero, we are done.
\end{proof}

Henceforth, we will assume that the conditions of Claim \ref{off}  hold.

\subsubsection{A final (horizontal) adjustment.}

The images $\tilde{\Psi}(L_{1,1})$, $\tilde{\Psi}(L_{1,2})$ and $\tilde{\Psi}
(L_{2,2})$ are disjoint from $S \times T$ but $\tilde{\Psi}$ still fixes $L_{1,2}$ which intersects $S \times T$. Since $L_{1,2}$ is close to the boundary of $S \times T$, we can make a simple adjustment to obtain the desired map $\Psi$ which  moves $L_{1,2}$ off of $S \times T$ as well.

Let $f \colon \RR/\ZZ \to \RR$ be a smooth function such that for some positive real number $c(f)$ greater than and arbitrarily close to $\ell+\sigma -2$ we have $$f'(s) = c(f),\, \text{for} \,\,s \in \left[0, \frac{1+ \delta}{\ell}\right],$$ $\max(f') = c(f)$,  and $\min(f')$ is less than and arbitrarily close to
$$-\frac{c(f)}{\frac{\ell}{1+ \delta}-1}.$$
Setting $F(\theta_1, \theta_2) =f(\theta_1)$, we have
\begin{equation*}
\label{fflow}
\phi^t_F(\theta_1,R_1, \theta_2, R_2) = (\theta_1,R_1+tf'(\theta_1), \theta_2 , R_2).
\end{equation*}

Our lower bound for $c(f)$ implies that $\phi^1_F(L_{2,1})$ is disjoint from $S \times T$. Looking at the $R_2$-component, it is clear that $\phi^1_F(L_{2,1})$ is disjoint from $L_{2,2}= \tilde{\Psi}(L_{2,2})$. To prove that  $\phi^1_F(L_{2,1})$ is also disjoint from $\tilde{\Psi}(L_{1,1})$ and $\tilde{\Psi}(L_{1,2})$ it suffices to prove the following.

\begin{lemma}
The sets $\{R_1 \leq 1+ c(g)\}$ and $\phi^1_F(L_{2,1})$ are disjoint.
\end{lemma}

\begin{proof}
It suffices to prove that 
\begin{equation*}
\label{ }
2-\frac{c(f)}{\frac{\ell}{1+ \delta}-1} > 1+ c(g)
\end{equation*}
or, even more, that
\begin{equation*}
\label{ }
1> c(g) + \frac{\ell+\sigma -2}{\frac{\ell}{1+ \delta}-1}.
\end{equation*}
The latter inequality clearly holds for all sufficiently small values of $c(g)$ and $\ell+\sigma -2$.
\end{proof}

The Hamiltonian diffeomorphism 
$$\Psi= \phi^1_{F_{[L_{2,1}]}} \circ \phi^{1-\tau}_{H_{\left[\phi^{\tau}_{H}(L_{1,1}) \cup \phi^1_{G}(L_{1,2})\right]} }\circ \phi^{\tau}_{H_{[L_{1,1}]}} \circ \phi^1_{G_{[L_{1,2}]}}$$
now has all the desired properties. With its construction, the proof of Theorem \ref{eps} is complete. 

\begin{question}
Can $\Psi$, or any other Hamiltonian diffeomorphism which displaces the $L_{k,l}$ from $S \times T$, be generated by an autonomous Hamiltonian?
\end{question}


\begin{thebibliography} {99}



\bibitem{behwz}
F. Bourgeois, Y. Eliashberg, H. Hofer, K. Wysocki, and E. Zehnder,
 Compactness results in symplectic field theory,
{\em Geom. Topol.}, 7 (2003), 799--888.


\bibitem{bu} O. Buse, negative inflation and stability in symplectomorphism groups of ruled surfaces, {\em J. Symplectic Geometry}, {\bf 9} (2011), 147--160.


\bibitem{cm} K. Cieliebak and K. Mohnke, Punctured holomorphic curves and Lagrangian embeddings, {\em Inventiones Mathematicae}, 212 (2018), 213--295.

\bibitem{cisc} K. Cieliebak and M. Schwingenheuer, Hamiltonian unknottedness of certain monotone Lagrangian tori in $S^2\times S^2$, {\em Pacific Journal of Mathematics}, {\bf 299} (2019 ), 427--468. 






\bibitem{evle} J. Evans and Y. Lekili,  Generating the Fukaya categories of Hamiltonian $ G$-manifolds, {\em J. Amer. Math. Soc.} \textbf{32} (2019), 119-162 .





\bibitem{rgi} G. Dimitroglou-Rizell, E. Goodman and A. Ivrii, Lagrangian isotopy of tori in $S^2 \times S^2$ and $\CC P^2$, {\it Geometric and Functional Analysis} 26 (2016), 1297--1358.

\bibitem{gr} M. Gromov, Pseudo-holomorphic curves in symplectic manifolds, {\em Inventiones Mathematicae}, 82 (1985), 307--347.







\bibitem{hl} R. Hind and S. Lisi, Symplectic embeddings of polydisks, {\it Selecta Math.}, 21 (2015), 1099--1120.

\bibitem{hl-err} R. Hind and S. Lisi, Erratum to: Symplectic embeddings of polydisks, {\it Selecta Math.}, 23 (2017), 813--815.

\bibitem{ho} R. Hind and E. Opshtein, Squeezing Lagrangian tori in dimension $4$, {\it Comment. Math. Helv.}, 95 (2020), 535--567.




\bibitem{hofa} H. Hofer, K. Wysocki and E. Zehnder, Properties of pseudoholomorphic curves in symplectisations I: Asymptotics, {\it Ann. Inst. H. Poincar\'{e} Anal. Non Lineaire}, 13 (1996) , 337--379.











\bibitem{ivriit} A. Ivrii, {\it Lagrangian unknottedness of tori in certain symplectic 4-manifolds}, PhD thesis, Stanford University, 2003.

\bibitem{lm96} F. Lalonde and D. McDuff, J-curves and the classification of rational and ruled symplectic 4-manifolds, {\it Contact and Symplectic Geometry}, Cambridge 1994, 3--42. 


\bibitem{maksmi} C. Mak and I. Smith, Non-displaceable Lagrangian links in four manifolds, {\it Geometric and Functional Analysis}, 31 (2021), 438--481.


\bibitem{mcd01} D. McDuff, Symplectomorphism groups and almost complex structures, {\it Enseignement Math}, (2001), 1--30.

\bibitem{mcp} D. McDuff and L. Polterovich, Symplectic packings and algebraic geometry, {\it Inventiones Math.}, \textbf{115} (1994), 405--429.

\bibitem{mcs} D. McDuff and D. Salamon, $J$-holomorphic curves and
symplectic topology. American Mathematical Society Colloquium
Publications, 52. American Mathematical Society, Providence, RI,
2004.






\bibitem{polshe} L. Polterovich and E. Shelukhin, Lagrangian configurations and Hamiltonian maps, arXiv:2102.06118v3.

\bibitem{oh} Y.-G. Oh, Addendum to Floer Cohomology of Lagrangian Intersections and Pseudo-Holomorphic Discs, I, {\it Comm. Pure Appl. Math.}, {\bf48} (1995), 1299--1302. 













\bibitem{vianna} R. Vianna, Infinitely many monotone Lagrangian tori in del Pezzo surfaces, {\em Selecta Mathematica}, (2016).


\bibitem{we} C. Wendl, Automatic transversality and orbifolds
of punctured holomorphic curves in dimension four,  {\it Comment. Math. Helv.}, 85 (2010), 347--407.




\end{thebibliography}
\end{document}